\documentclass[11pt]{article}

\usepackage{amsthm}
\usepackage{amsfonts}
\usepackage{amsmath}
\usepackage{amssymb}
\usepackage{colonequals}
\usepackage{enumerate}
\usepackage{framed}
\usepackage{graphicx}
\usepackage{setspace}
\usepackage{tikz-cd}

\newcommand{\pF}{\ensuremath{\partial_F}}
\newcommand{\pL}{\ensuremath{\partial_L}}
\newcommand{\pu}{\ensuremath{\partial_{L,u}}}

\newcommand{\what}[1]{\widehat{#1}}
\newcommand{\wtilde}[1]{\widetilde{#1}}
\newcommand{\ol}[1]{\overline{#1}}

\newcommand{\ve}{\ensuremath{\varepsilon}}

\newcommand{\cA}{\ensuremath{\mathcal{A}}}
\newcommand{\cF}{\ensuremath{\mathcal{F}}}
\newcommand{\g}{\ensuremath{\mathfrak{g}}}
\newcommand{\cO}{\ensuremath{\mathcal{O}}}
\newcommand{\R}{\ensuremath{\mathbb{R}}}
\newcommand{\s}[1]{\ensuremath{\mathbf{#1}}}
\newcommand{\cU}{\ensuremath{\mathcal{U}}}
\newcommand{\U}{\ensuremath{\mathbb{U}}}
\newcommand{\cX}{\ensuremath{\mathcal{X}}}

\newtheorem{theorem}{Theorem}[section]
\newtheorem{theorem*}{Theorem}
\newtheorem{lemma}[theorem]{Lemma}

\newtheorem{corollary}[theorem]{Corollary}

\theoremstyle{definition}
\newtheorem{definition}[theorem]{Definition}

\newtheorem*{remark}{Remark}

\newcommand{\TheTitle}{The Discrete-Time Geometric Maximum Principle}

\date{November 1, 2016}

\title{\TheTitle} 

\author{
  Robert Kipka\thanks{Department of Mathematical Sciences, Lake Superior State University, Sault Ste Marie, MI 49783, USA
    (\texttt{rkipka@lssu.edu}).}
  \and
  Rohit Gupta\thanks{Institute for Mathematics and its Applications, University of Minnesota, Minneapolis, MN 55455, USA
    (\texttt{gupta311@umn.edu}).}
}

\begin{document}

\maketitle

\begin{abstract}
We establish a variety of results extending the well-known Pontryagin maximum principle of optimal control to discrete-time optimal control problems posed on smooth manifolds. These results are organized around a new theorem on critical and approximate critical points for discrete-time geometric control systems. We show that this theorem can be used to derive Lie group variational integrators in Hamiltonian form; to establish a maximum principle for control problems in the absence of state constraints; and to provide sufficient conditions for exact penalization techniques in the presence of state or mixed constraints. Exact penalization techniques are used to study sensitivity of the optimal value function to constraint perturbations and to prove necessary optimality conditions, including in the form of a maximum principle, for discrete-time geometric control problems with state or mixed constraints.
\end{abstract}



\section{Introduction}

The celebrated Pontryagin maximum principle \cite{pontryagin1962} is a powerful tool for analyzing continuous-time optimal control problems for finite-dimensional nonlinear control systems. Since its discovery in the late 1950s, a considerable amount of effort has been spent in extending the maximum principle in various directions, which has resulted in a significant amount of literature. We briefly mention two such directions: the first considers the setting of smooth manifolds instead of Euclidean spaces (see, for example, \cite{AS2004,barbero2009, kipka-ledyaev-2014a, kipka-ledyaev-2015a, sussmann1997}) while the second considers discrete-time optimal control problems for finite-dimensional nonlinear control systems (see, for example, \cite{boltyanskii1978, mordukhovich2006b} or \cite{bourdin2013} and references therein). As discussed in \cite{bourdin2013}, the formulation of the maximum principle for discrete-time optimal control problems becomes quite tricky since the maximization condition cannot be expected to hold in general and in fact some of the early literature on this topic was mathematically incorrect (see \cite{boltyanskii1978, mordukhovich2006b} for a number of interesting counterexamples). However, the maximization condition does hold under appropriate convexity assumptions on the dynamics and moreover an approximate maximization condition can be derived (under suitable assumptions) in the absence of such assumptions (see \cite[Section 6.4]{mordukhovich2006b}).

It is quite surprising to note that there has been almost no investigation into a maximum principle for discrete-time optimal control problems defined on manifolds despite the fact that smooth manifolds arise quite naturally in many practical control problems such as robotics (see, for example, \cite{kobilarov2007, kobilarov2011} and references therein) and spacecraft attitude control (see, for example, \cite{gupta2015, kalabic2016, phogat2015} and references therein).

Indeed the interested reader can find a wide variety of geometric control problems in texts such as \cite{AS2004,bloch2003,bullo-lewis2005,jurdjevic1997geometric}.

With this motivation, in this paper we obtain necessary conditions for optimality for discrete-time optimal control problems defined on manifolds and thereby arrive at a maximum principle for such problems.

The maximum principle is in fact a necessary condition for a control to be \emph{critical} for a cost function, though such a control may not be minimizing. In many cases it is useful to consider necessary conditions for a control to be $\Delta$-\emph{critical} for $\Delta \ge 0$. In particular, this paper is concerned with careful analysis of $\Delta$-critical points for a function
\begin{equation}
\label{eq:jdef}
  J(\s{u}) = \ell(q_n) + \sum_{i = 0}^{n-1} L_i(q_i,u_i),
\end{equation}
where $(u_i)_{i = 0}^{n-1}$ is a control sequence generating a state sequence $(q_i)_{i = 0}^n$ through a discrete-time geometric control system (see Definition \ref{defn:dgcs}).

For $\Delta = 0$ such controls are critical points of $J$ in a classical sense. For $\Delta > 0$ the notion of $\Delta$-critical controls is closely related to that of \emph{strong slope} introduced in \cite{de1980problems}. The reader can find a number interesting applications of strong slope in \cite{guler2010foundations} and our definition of $\Delta$-critical, below in Definition \ref{defn:critical}.

\subsection{Organization and Contributions}

This paper is organized as follows. In the section following we establish novel necessary conditions for $\Delta$-critical controls: Theorem \ref{thm:dmp}. As first applications of Theorem \ref{thm:dmp} we ($i$) derive a discrete-time geometric maximum principle applicable to a wide variety of control problems, including those arising from discretization schemes such as those of Hans Munthe-Kaas \cite{munthe1999high} and ($ii$) relate critical controls to structure-preserving variational integrators for Lie groups \cite{lee2005lie}. This section concretely demonstrates that the maximum principle for discrete-time optimal control problems and the Hamiltonian formulation of variational integrators are different manifestations of the same phenomenon.

We then turn to questions of exact penalization for constrained problems on manifolds, beginning in Section \ref{sec:decrease-princ} with a derivation of a decrease principle in a geometric setting. Such principles are well studied and of much use in optimization \cite{clarke2013, penot2013} and indeed have been considered before in the geometric setting \cite{azagra2007applications}. The main result of this section, Theorem \ref{thm:penalty}, is closely related to the decrease condition and solvability theorem of \cite{azagra2007applications} (see \cite{clarkeand1998} for proofs in the Hilbert space setting). We provide a novel proof in terms of Fr\'echet subgradient without the assumption that the underlying Riemannian manifold is complete. In addition, while the propositions of \cite{azagra2007applications} require local assumptions on the norm of the subgradient, Theorem \ref{thm:penalty} requires assumptions only on a closed set and its Clarke tangent cone. Finally, the conditions of Theorem \ref{thm:penalty} are given in terms of directional derivatives rather than norms of subdifferentials. As such these conditions have a direct interpretation in terms of discrete-time controllability.

Following this we introduce a new constraint qualification, \emph{strict $\varphi_i$-normality}, for discrete-time geometric control problems. We show in Section \ref{sec:decrease-princ} that strict $\varphi_i$-normality may be used to state sufficient conditions for exact penalization and for calmness of the value function in a geometric setting.

In the final section of the paper we apply techniques of exact penalization to derive maximum principles for discrete-time geometric control problems with either pure state or mixed constraints.  As an illustration of these techniques, we conclude the paper with a discussion of a discrete-time geometric optimal control problem with mixed constraints given by a family of smooth inequality constraints $g_j(q_i, u_i) \le 0$, although we emphasize that the results developed below are quite general and are by no means limited to inequality constraints.

Before continuing we provide the main definitions, notations, and standing assumptions used in this paper.

\subsection{Definitions, Notations, and Standing Assumptions}

\begin{definition}
\label{defn:dgcs}
  A \emph{discrete-time geometric control system} is a collection of finite-\\dimensional manifolds $Q$ and $\left\{U_i\right\}_{i = 0}^{n-1}$, not necessarily of the same dimension; a finite collection of closed sets $\U_i \subseteq U_i$; and a finite collection $\left\{F_i\right\}_{i = 0}^{n-1}$ of mappings $F_i : Q \times U_i \rightarrow Q$.
\end{definition}

We write $\cU \subset \prod_{i = 0}^{n-1} U_i$ for the set of all $\s{u}$ satisfying $u_i \in \U_i$ for $0 \le i \le n-1$.

\begin{definition}
A sequence $\s{u} \in \cU$ is said to be a \emph{control sequence}.
\end{definition}

\begin{definition}
Given a control sequence $\s{u}$ and initial state $q_0$, a \emph{state sequence} is the sequence $(q_i)_{i = 0}^n$ determined for $1 \le i \le n$ through
\begin{equation}
\label{eq:update-rule}
  q_i = F_{i-1}(q_{i-1}, u_{i-1}).
\end{equation}
\end{definition}

Throughout this paper we denote sequences using bold so that, for example, $\s{q}$ denotes a sequence $\s{q} = \left(q_i\right)_{i = 0}^n$.
A sequence of sequences will be denoted $\left(\s{q}_k\right)_{k = 1}^\infty$ so that, for example, $\s{q}_k = \left(q_{k,i}\right)_{i = 0}^n$. Letters $q,r$ refer in all cases to states and $u,c$ to controls.

We typically study cost functions $J : \cU \rightarrow \R$ for fixed $q_0$ as defined by \eqref{eq:jdef}, although variations in $q_0$ are considered in Section \ref{subsec:initial-variation}. The following assumptions are in place throughout the paper unless explicitly stated otherwise.

\textbf{Standing Assumptions:} \emph{Functions $\ell$ and $L_i$ in \eqref{eq:jdef} are locally Lipschitz; the maps $F_i$ in \eqref{eq:update-rule} are $C^1$-smooth; and the manifolds $U_i$ are Riemannian with metric $g_i$.}

We remark that the assumption of a Riemannian structure on $U_i$ is made without loss of generality. These metrics induce a product metric $g$ on $\prod_{i = 0}^{n-1}U_i$ through
\begin{equation*}
  g(\s{v}, \s{w}) \colonequals \sum_{i = 0}^{n-1} g_i(v_i, w_i).
\end{equation*}

Finally, we will denote the pushforward of a map $F : M \rightarrow N$ through $DF(q) : T_qM \rightarrow T_{F(q)}N$ and the pullback through $DF(q)^* : T^*_{F(q)}N \rightarrow T^*_qM$. Often we will have need to write $D_q F(q,u) : T_q Q \rightarrow T_{F(q,u)}Q$ for the partial derivative of $F$ with respect to $q$. Likewise, we may write $d_q L_i(q,u) \in T_q^* Q$ for the partial exterior derivative of $L_i : Q \times U_i \rightarrow \R$.

\subsubsection{Nonsmooth Analysis}

Techniques of nonsmooth or variational analysis play a central role in this paper. Here we provide definitions for the particular subgradients, normal cones, and tangent cones used in the paper. A useful introduction to the techniques of nonsmooth analysis on smooth manifolds can be found in \cite{ledyaevzhu2007}. We recommend \cite{clarke2013} or \cite{clarkeand1998} for an introduction to nonsmooth analysis in the context of optimization and control. We also mention as useful references the books \cite{borweinzhu2005,schirotzek2007} and the comprehensive volumes \cite{mordukhovich2006a,mordukhovich2006b}.

\begin{definition}
\label{defn:lipschitz}
A function $f : M \rightarrow \R$ is \emph{locally Lipschitz} at $q \in M$ if there exists a coordinate chart $\varphi : M \rightarrow \R^d$ whose domain $\cO$ includes $q$ such that the function $f \circ \varphi^{-1} : \R^d \rightarrow \R$ is Lipschitz on $\varphi(\cO)$.
\end{definition}

\begin{remark}
  When $M$ is Riemannian with distance function $d$, Definition \ref{defn:lipschitz} is equivalent to the usual metric space definition and we use the two interchangeably in this case.
\end{remark}

\begin{definition}
For $f : M \rightarrow \R$ locally Lipschitz, the \emph{lower Dini derivative} is defined, for $v \in T_qM$, by
\begin{equation*}
\underline{D}f(q;v) \colonequals  \liminf_{\lambda \downarrow 0} \frac{f(c_v(\lambda)) - f(q)}{\lambda}
\end{equation*}
where $c_v : \R \rightarrow M$ is any smooth curve satisfying $c_v^\prime(0) = v$.
\end{definition}

\begin{definition}
A function $f : M \rightarrow \R \cup \left\{\infty\right\}$ is \emph{lower semicontinuous} at $q \in M$ if for any sequence $(q_n)_{n = 1}^\infty$ converging to $q$ there holds
\begin{equation*}
f(q) \le  \liminf_{n \rightarrow \infty} f(q_n).
\end{equation*}
\end{definition}

\begin{definition}
\label{defn:frechet}
  A covector $p \in T_q^*M$ is a \emph{Fr\'echet subgradient} for a lower semicontinuous function $f$ at $q$ if there exists a $C^1$-smooth function $g : M \rightarrow \R$ such that
   $p = dg(q)$ and $f - g$ has a local minimum at $q$. The \emph{Fr\'echet subdifferential} $\pF f(q)$ is the (possibly empty) set of all such covectors.
\end{definition}

\begin{definition}
\label{defn:limiting}
  A covector $p \in T_q^*M$ is a \emph{limiting subgradient} for a lower semicontinuous function $f$ at $q$ if there exist sequences $q_n \rightarrow q$ and $p_n \in \pF f(q_n)$ such that $f(q_n) \rightarrow f(q)$ and $p_n \rightarrow p$.
  The \emph{limiting subdifferential} $\pL f(q)$ is the (possibly empty) set of all such covectors.
\end{definition}

Occasionally we will have need for the partial limiting subgradient $\pu f(q,u)$, which is simply the limiting subgradient of the function $u \mapsto f(q,u)$.

\begin{definition}
\label{defn:limiting-normal}
  Given a closed set $S \subset M$, the \emph{limiting normal cone} to $S$ at $s \in S$ is the set $N_S^L(s) \colonequals \pL \chi_S(s)$, where $\chi_S : M \rightarrow \R \cup \left\{\infty\right\}$ is the function
  \begin{equation*}
    \chi_S(s) \colonequals \begin{cases} \hspace{4pt} 0 &s \in S \\
    \hspace{1.5pt} \infty &s \not \in S. \end{cases}
  \end{equation*}
\end{definition}

\begin{definition}
\label{defn:clarke-tangent}
The \emph{Clarke tangent cone} $T_S^C(q)$ to $S$ at $q$ is the polar of the limiting normal cone:
\begin{equation*}
  T_S^C(q) \colonequals \left\{v \in T_q M \, : \, \left<p,v\right> \le 0 \; \mathrm{for} \; \mathrm{all} \; p \in N_S^L(q)\right\}.
\end{equation*}
\end{definition}

It can be useful to study $\pL f(q)$, $N_S^L(q)$, and $T_S^C(q)$ in local coordinates. In this direction we mention \cite[Theorem 4.1]{ledyaevzhu2007}:
\begin{theorem}
 \label{thm:subgrad-invariance}
 Let $f : \R^d \rightarrow \R \cup \left\{ \infty\right\}$ be lower semicontinuous, $\cO \subset M$ an open set, and $\varphi : \cO \rightarrow \R^d$ a $C^1$-smooth diffeomorphism. If $x\colonequals \varphi(q)$ then
  \begin{equation}
  \label{eq:subgrad-invariance}
    \pL (f \circ \varphi)(q) = \varphi^* \pL f(x).
  \end{equation}
\end{theorem}

If we are given a closed set $S \subset M$, a point $q \in S$, and a coordinate chart $\varphi : M \rightarrow \R^d$ whose domain $\cO$ includes $q$ then we can write $x \colonequals \varphi(q)$ and obtain from \eqref{eq:subgrad-invariance} the useful formula
\begin{equation}
\label{eq:local-normal}
  N_S^L(q) = \varphi^* N_{\varphi(S \cap \cO)}^L(x).
\end{equation}
 Although in general $\varphi(S \cap \cO) \subset \R^d$ may not be a closed set, Definition \ref{defn:limiting-normal} continues to make sense in \eqref{eq:local-normal} because the function $\chi_{\varphi(S \cap \cO)}$ is lower semicontinuous on a neighborhood of $x$.

Likewise, the reader may wish to check the dual formula
\begin{equation}
  \label{eq:local-tangent}
  \varphi_* T_S^C(q) = T_{\varphi(S \cap \cO)}^C(x),
\end{equation}
which is itself a consequence of \eqref{eq:local-normal} and Definition \ref{defn:clarke-tangent}. As a consequence of \eqref{eq:local-tangent}, the Clarke tangent cone retains on manifolds many of the attractive features developed in \cite{clarkeand1998} for Banach spaces.

Finally, we will have need for the following theorem, which is a geometric version of a result due to Subbotin \cite{subbotin1991property}:
\begin{theorem}
\label{thm:subbotin}
  Let $V \subset T_qM$ be a compact, convex set and $f : M \rightarrow \R$ locally Lipschitz. For any
  \begin{equation*}
    \rho \le \inf_{v \in V} \underline{D}f(x;v)
  \end{equation*}
  there exists $p \in \pL f(x)$ such that
  \begin{equation*}
    \rho \le \inf_{v \in V} \left< p,v\right>.
  \end{equation*}
\end{theorem}

\begin{proof}
  A version of this Theorem for $M = \R^d$ can be obtained by combining Theorem 3.4.2 and Proposition 3.4.5 in \cite{clarkeand1998}. The general manifold case can then be obtained using the local coordinate formula of Theorem \ref{thm:subgrad-invariance}.
\end{proof}

\section{Necessary Conditions for Approximate Critical Points}

We turn now to the first result of this paper, a necessary condition for control $\s{u}$ to be approximately critical in the following sense:

\begin{definition}
\label{defn:critical}
  Control $\s{u}$ is \emph{$\Delta$-critical} ($\Delta \ge 0$) for function $J : \cU \rightarrow \R$ if for any $\s{v} \in T_{\cU}^C(\s{u})$ there holds
\begin{equation}
\label{eq:defn-critical}
 - \Delta \left\|\s{v} \right\|_g \le  \underline{D}J(\s{u}; \s{v}).
\end{equation}
In the case where $\Delta = 0$ we simply say that $\s{u}$ is \emph{critical}.
\end{definition}

We mention that in the theorem following \eqref{eq:endpoint-condition} is in analogy with the transversality condition of the classical maximum principle; \eqref{eq:discrete-adjoint} with the adjoint equations; and \eqref{eq:discrete-amp} with the maximum principle and indeed our proof uses techniques related to those in \cite{kipka-ledyaev-2014a}. These connections are made stronger in later sections.

\begin{theorem}
\label{thm:dmp}
If $\s{u} \in \cU$ is $\Delta$-critical for $J$ then, under the standing assumptions, there exist sequences $\s{a} = (a_i)_{i = 1}^n$ and $\s{b} = (b_i)_{i = 0}^{n-1}$ with $a_n \in \pL \ell(q_n)$, $b_0 \in \pu L_0(q_0, u_0)$, and $(a_i, b_i) \in \pL L_i(q_i, u_i)$ which determine a sequence of costates $p_i \in T_{q_i}^*Q$ for $1 \le i \le n$ through
\begin{align}
\label{eq:endpoint-condition}
-p_n & = a_n \in \pL \ell(q_n)\\
\label{eq:discrete-adjoint}
p_{i-1} &= -a_{i-1} + D_q F_{i-1}(q_{i-1}, u_{i-1})^* p_i \hspace{10pt} (2 \le i \le n)
\end{align}
satisfying for all $v \in T_{\U_i}^C(u_i)$
\begin{equation}
\label{eq:discrete-amp}
 - \Delta \left\|v\right\|_g \le \left<b_i - D_u F_i(q_i,u_i)^* p_{i+1}, v \right> \hspace{10pt} (0 \le i \le n-1).
\end{equation}
\end{theorem}

\begin{proof}
With $\s{u}$, $\Delta$, and $J$ as above we denote by $\s{q}$ the state sequence associated with $\s{u}$. Fix a vector $\s{v} \in T_{\cU}^C(\s{u})$, let $\s{u}(\lambda)$ denote a smooth curve taking values in $\prod_{i = 0}^{n-1} U_i$ and satisfying $\s{u}^\prime(0) = \s{v}$, and write $\s{q}(\lambda)$ for the corresponding state sequence.

We introduce, for $0 \le i < j \le n$, mappings $\cF_{i,j} : T_{q_i}M \rightarrow T_{q_j}M$ defined by
\begin{equation*}
  \cF_{i,j} \colonequals D_qF_{j-1}(q_{j-1}, u_{j-1}) \circ \dots \circ D_q F_{i}(q_{i}, u_{i})
\end{equation*}
 We let $\cF_{i,i}$ denote the identity map on $T_{q_i}M$ so that $\cF_{i,j}$ is defined for $0 \le i \le j \le n$.

 \begin{lemma}
 With $q_j(\lambda)$ and $\s{v} = (v_i)_{i = 0}^{n-1} \in T_{\cU}^C(\s{u})$ defined as above there holds
 \begin{equation}
\label{eq:flow-derivative}
q_j^\prime(0) = \sum_{i = 0}^{j-1} \cF_{i+1,j} D_uF_i(q_i,u_i)v_i.
\end{equation}
\end{lemma}

\begin{proof}
We leave it to the reader to check that \eqref{eq:flow-derivative} holds for $j = 0$ and $j = 1$ and we proceed to prove the formula for $j > 1$ through induction.
Fix $j > 1$ and suppose that \eqref{eq:flow-derivative} holds for $j-1$. Then
\begin{equation*}
\begin{aligned}
q_j^\prime(0)  & = \left.  \frac{d }{d \lambda} \right|_{\lambda = 0} F_{j-1}(q_{j-1}(\lambda),u_{j-1}(\lambda)) \\
& = D_q F_{j-1}(q_{j-1}, u_{j-1}) q_{j-1}^\prime(0)  + D_u F_{j-1}(q_{j-1}, u_{j-1})v_{j-1}.
\end{aligned}
\end{equation*}
From the induction hypotheses there now follows
\begin{equation}
\label{eq:variational-step}
\begin{aligned}
 q_j^\prime(0) & = D_q F_{j-1}(q_{j-1}, u_{j-1})  \sum_{i = 0}^{j-2} \cF_{i+1,j-1} D_uF_i(q_i,u_i)v_i  + D_u F_{j-1}(q_{j-1}, u_{j-1})v_{j-1} \\
& = \sum_{i = 0}^{j-2} \cF_{j-1,j}\cF_{i+1,j-1} D_uF_i(q_i,u_i)v_i +\cF_{j,j} D_u F_{j-1}(q_{j-1}, u_{j-1})v_{j-1}.
\end{aligned}
\end{equation}
Using the definition one may check that the maps $\cF_{i,j}$ satisfy the following semigroup law:
\begin{equation}
  \label{eq:forward-semigroup}
  \cF_{j,k} \circ \cF_{i,j} = \cF_{i,k} \hspace{10pt} (0 \le i \le j \le k \le n).
\end{equation}
Thus \eqref{eq:variational-step} can be simplified to
\begin{equation*}
 q_j^\prime(0) = \sum_{i = 0}^{j-2} \cF_{i+1,j} D_uF_i(q_i,u_i)v_i +\cF_{j,j} D_u F_{j-1}(q_{j-1}, u_{j-1})v_{j-1}
\end{equation*}
and this is the same as \eqref{eq:flow-derivative}.
\end{proof}

We now consider the manifold
\begin{equation*}
  M \colonequals U_0 \times \dots \times U_{n-1} \times \underbrace{Q \times \dots \times Q}_{n\text{-}\mathrm{copies}}
\end{equation*}
and we denote elements of $M$ through $(\s{c}, \s{r}) = (c_0,\dots, c_{n-1}, r_1, \dots , r_n)$ in order to match indices on controls and states. Let $\mathcal{J} : M \rightarrow \R$ be the function
\begin{equation*}
  \mathcal{J}(\s{c}, \s{r}) = \ell(r_n) + L_0(q_0, c_0) + \sum_{i = 1}^{n-1} L_i(r_i,c_i)
\end{equation*}
and let $V \subseteq T_{(\s{u},\s{q})}M$ denote the set
\begin{equation*}
V\colonequals\left\{ \left(v_0, \dots, v_{n-1}, q_1^\prime(0), \dots ,q_n^\prime(0) \right) \, : \, \s{v} \in T_{\cU}^C(\s{u}), \; \left\|\s{v}\right\|_g \le 1\right\},
\end{equation*}
where $(q_i(\lambda))_{i = 1}^n$ is the variation of the state $\s{q}$ corresponding to $\s{v}$. We write $(\s{v},\s{w})$ for any such vector and index $(\s{v},\s{w})$ as $(v_0,\dots, v_{n-1}, w_1, \dots, w_n)$.

Because $\mathcal{J}$ is locally Lipschitz, \eqref{eq:defn-critical} implies that
\begin{equation*}
- \Delta \le \inf_{(\s{v},\s{w}) \in V} \underline{D}\mathcal{J}(\s{u}, \s{q} ;\s{v}, \s{w})
\end{equation*}
and this formula holds for any $(\s{v}, \s{w}) \in V$.

Moreover, $T_{\cU}^C(\s{u})$ is closed and convex and so formula \eqref{eq:flow-derivative} with the constraint $\left\|\s{v}\right\|_g \le 1$ imply that $V$ is both compact and convex.
Applying Theorem \ref{thm:subbotin} we obtain $(\s{b}, \s{a}) \in \pL \mathcal{J}(\s{u}, \s{q})$ for which
\begin{equation*}
 -\Delta \le \inf_{(\s{v},\s{w}) \in V} \left<(\s{b}, \s{a}), (\s{v},\s{w}) \right>.
\end{equation*}
Indexing $(\s{b}, \s{a})$ as $(b_0, \dots, b_{n-1}, a_1, \dots, a_n)$ we check that $b_0 \in \pu L_0(q_0, u_0)$, $a_n \in \pL \ell(q_n)$, $(a_i, b_i) \in \pL L_i(q_i, u_i)$ for $1 \le i \le n-1$, and for $\s{v} \in T_{\cU}^C(\s{u})$ with $\left\|\s{v}\right\|_g \le 1$
\begin{equation*}
-\Delta \le \sum_{i = 0}^{n-1} \left<b_i, v_i \right> + \sum_{j = 1}^n \left< a_j, q_j^\prime(0) \right>.
\end{equation*}
Using \eqref{eq:flow-derivative} we write this as
\begin{equation}
\label{eq:unit-vector-step-pmp}
-\Delta \le \sum_{i = 0}^{n-1} \left<b_i, v_i \right> + \sum_{j = 1}^n  \sum_{i = 0}^{j-1}  \left< a_j,  \cF_{i+1,j}D_u F_i(q_i,u_i)v_i \right>.
\end{equation}

Since \eqref{eq:unit-vector-step-pmp} holds for all unit vectors in the cone $T_{\cU}^C(\s{u})$, the inequality
\begin{equation*}
-\Delta \left\|\s{v}\right\|_g \le \sum_{i = 0}^{n-1} \left<b_i, v_i \right> + \sum_{j = 1}^n  \sum_{i = 0}^{j-1}  \left< a_j,  \cF_{i+1,j}D_u F_i(q_i,u_i)v_i \right>
\end{equation*}
holds for all vectors $\s{v} \in T_{\cU}^C(\s{u})$.
Rearranging the double sum we obtain
\begin{equation*}
-\Delta \left\|\s{v}\right\|_g \le \sum_{i = 0}^{n-1} \left<b_i, v_i \right> + \sum_{i= 0}^{n-1}  \sum_{j = i+1}^{n}  \left< a_j,  \cF_{i+1,j}D_u F_i(q_i,u_i)v_i \right>
\end{equation*}
thus proving that for all $\s{v} \in T_{\cU}^C(\s{u})$
\begin{equation}
\label{eq:amess}
-\Delta \left\|\s{v}\right\|_g \le \sum_{i = 0}^{n-1} \left<b_i, v_i \right> + \sum_{i= 0}^{n-1} \left< D_u F_i(q_i,u_i)^*\sum_{j = i+1}^{n}  \cF_{i+1,j}^* a_j,  v_i \right>.
\end{equation}

We now define a sequence $(p_i)_{i =1}^{n}$ through $p_i = -\displaystyle \sum_{j = i}^n \cF_{i,j}^* a_j$ so that \eqref{eq:amess} is
\begin{equation*}
-\Delta \left\|\s{v}\right\|_g \le \sum_{i = 0}^{n-1} \left<b_i-D_u F_i(q_i,u_i)^*p_{i+1},  v_i \right>.
\end{equation*}
Recalling that $T_{\cU}^C(\s{u}) = T_{\U_0}^C(u_0) \times T_{\U_1}^C(u_1) \times \dots \times T_{\U_{n-1}}^C(u_{n-1})$, which follows from \eqref{eq:local-tangent} and \cite[page 85]{clarkeand1998},
and choosing the vector $(0, \dots, 0, v, 0, \dots, 0)$ for arbitrary $v \in T_{\U_i}^C(u_i)$ we obtain \eqref{eq:discrete-amp}.

We note that $-p_n = a_n \in \pL \ell(q_n)$ and for $1 \le i \le n$ we have $p_i \in T_{q_i}^*M$. All that remains is to check that \eqref{eq:discrete-adjoint} holds. For this we note that \eqref{eq:forward-semigroup} implies
\begin{equation}
\label{eq:semigroup}
 \cF_{i,j}^* \circ \cF_{j,k}^* = \cF_{i,k}^* \hspace{10pt} (0 \le i \le j \le k \le n)
\end{equation}
For any $2 \le i \le n$, the definition of $p_i$ and the relation \eqref{eq:semigroup} imply the equalities
\begin{align*}
  p_{i-1} &= -\sum_{j = i-1}^n \cF_{i-1,j}^* a_j = -a_{i-1} - \cF^*_{i-1,i} \sum_{j = i}^n \cF_{i, j}^* a_j \\
	&= - a_{i-1} + D_q F_{i-1}(q_{i-1}, u_{i-1})^* \left(-\sum_{j = i}^n \cF_{i, j}^* a_j \right)
\end{align*}
and this is precisely \eqref{eq:discrete-adjoint}.
\end{proof}

\subsection{Varying the Initial State}
\label{subsec:initial-variation}

In certain applications it can be useful to allow variations in the initial state, $q_0$. Since we have assumed that the sets $U_i$ are manifolds, it is interesting to study this case by considering the initial state $q_0$ to be a control. In particular, let $J : Q \times \cU \rightarrow \R$ be defined by
 \begin{equation}
 \label{eq:cost-with-initial-condition}
   J(q_0,\s{u}) = \kappa(q_0) + \ell(q_n) + \sum_{i = 0}^{n-1} L_i(q_i, u_i),
 \end{equation}
 where $\kappa : Q \rightarrow \R$. We suppose that initial state $q_0$ is required to lie a closed set $S \subset Q$ and generalize Definition \ref{defn:critical} as follows:
 \begin{definition}
   \label{defn:general-critical}
   State $q_0$ and control $\s{u} \in \cU$ are \emph{$\Delta$-critical} ($\Delta \ge 0$) for function $J$ defined by \eqref{eq:cost-with-initial-condition} if for any $w \in T_S^C(q_0)$ and $\s{v} \in T_{\cU}^C(\s{u})$ there holds
\begin{equation}
\label{eq:gen-defn-critical}
 - \Delta \left(\left\|w\right\|^2+\left\|\s{v} \right\|^2\right)^{1/2} \le  \underline{D}J(q_0,\s{u};w, \s{v}).
\end{equation}
\end{definition}

\begin{theorem}
\label{thm:dmp-initial-variation}
If $q_0 \in Q$ and $\s{u} \in \cU$ are $\Delta$-critical for $J$ in the sense of Definition \ref{defn:general-critical}, the standing assumptions hold, and $\kappa : Q \rightarrow \R$ is locally Lipschitz, then there exist sequences $\s{a} = (a_i)_{i = 1}^n$ and $\s{b} = (b_i)_{i = 0}^{n-1}$ with $a_n \in \pL \ell(q_n)$, $b_0 \in \pu L_0(q_0, u_0)$, and $(a_i, b_i) \in \pL L_i(q_i, u_i)$ which determine a sequence of costates $p_i \in T_{q_i}^*Q$ ($0 \le i \le n$) through \eqref{eq:endpoint-condition}
\begin{equation*}
p_{i-1} = -a_{i-1} + D_q F_{i-1}(q_{i-1}, u_{i-1})^* p_i \hspace{10pt} (1 \le i \le n)
\end{equation*}
satisfying \eqref{eq:discrete-amp} for all $v \in T_{\U_i}^C(u_i)$. In addition, there exists $\beta \in \pL \kappa(q_0)$ such that for all $v \in T_S^C(q_0)$
\begin{equation}
\label{eq:p0-condition}
 - \Delta \left\|v\right\| \le \left<\beta - p_{0}, v \right>.
\end{equation}
\end{theorem}

\begin{proof}
Define a collection of Riemannian manifolds $\left\{U_i\right\}_{i = -1}^{n-1}$ by setting $U_{-1} = Q$ and $U_i = U$ for $0 \le i \le n-1$ and let $\what{F}_i : Q \times U_i \rightarrow Q$ be given by
\begin{equation*}
  \what{F}_{-1}(q,\wtilde{q}) \colonequals \wtilde{q}, \, \what{F}_i(q,u) \colonequals F_i(q,u) \hspace{10pt} (0 \le i \le n-1).
\end{equation*}
Let $\what{L}_i : Q \times U_i \rightarrow \R$ be
\begin{equation*}
 \what{L}_{-1}(q,\wtilde{q}) \colonequals \kappa(\wtilde{q}), \, \what{L}_i(q,u) \colonequals L_i(q,u) \hspace{10pt} (0 \le i \le n-1).
\end{equation*}
We agree to write $q_{-1} \colonequals q_0$ and $\what{\s{w}}$ for controls $(q,\s{w}) \in S \times \cU$.
One may then check that because $(q_0, \s{u})$ is $\Delta$-critical for the cost function \eqref{eq:cost-with-initial-condition} in the sense of Definition \ref{defn:general-critical}, the control $\what{\s{u}} \colonequals (q_0, \s{u})$ is $\Delta$-critical for
\begin{equation*}
  \what{J}(\what{\s{w}}) \colonequals \ell(q_n) + \sum_{i = -1}^{n-1} \what{L}_i(q_i, w_i)
\end{equation*}
in the sense of Definition \ref{defn:critical}. The result now follows from Theorem \ref{thm:dmp}.
\end{proof}

It's interesting to note that if $T_S^C(q_0) = T_{q_0}^*Q$, for example, if $q_0$ is in the interior of $S$, then \eqref{eq:p0-condition} simplifies to $p_0 \in \pL \kappa(q_0) + \Delta \ol{B}$, where $B$ is the open unit ball in $T_{q_0}^*Q$.

\section{Applications}

As first applications of Theorem \ref{thm:dmp} we consider in this section ($i$) applications to geometric control systems in the absence of state constraints and ($ii$) applications to variational integrators. A careful study of problems for which there are pure state or mixed constraints will be undertaken below, beginning in Section \ref{sec:decrease-princ}.

\subsection{A Discrete-Time Geometric Maximum Principle}
\label{sec:discrete-gmp}

In continuous-time,\\ the Pontryagin maximum principle is used to study trajectories of control systems
\begin{equation}
\label{eq:classical-control-system}
  \dot{q}(t) = f(q(t), u(t))
\end{equation}
in which $f : Q \times U \rightarrow TQ$ is a $C^1$-smooth map satisfying $f(q,u) \in T_qQ$ for all $(q,u) \in Q \times U$ \cite{AS2004, barbero2009, kipka-ledyaev-2014a, kipka-ledyaev-2015a, sussmann1997}. In this section we undertake a brief study of discrete-time geometric control systems which the update maps $F_i$ arise through discretization of \eqref{eq:classical-control-system}. Let us suppose we are given finite interval $\left[0,T\right]$ and positive numbers $(h_i)_{i = 0}^{n-1}$ for which $\sum_{i = 0}^{n-1} h_i = T$.

If $Q$ is a Riemannian manifold, then a natural generalization of the classical scheme
$x_i = x_{i-1} + h_{i-1} f(x_{i-1}, u_{i-1})$
is given by
\begin{equation*}
  q_i \colonequals \exp_{q_{i-1}}\left(h_{i-1} f(q_{i-1}, u_{i-1})\right).
\end{equation*}
In this case the update map $F_i : Q \times U \rightarrow Q$ factors as
\begin{equation}
\label{eq:riemannian-factorization}
  \begin{tikzcd}
    Q \times U \arrow[rr,"F_i"] \arrow{dr}[swap]{f_i} & & Q \\
    & TQ \arrow{ur}[swap]{\exp}
  \end{tikzcd}
\end{equation}

More generally, we may replace $\exp : TQ \rightarrow Q$ with a smooth map $E : TQ \rightarrow Q$. For example, because the Riemannian exponential can be expensive to calculate in practice, approximate exponentials $E : TQ \rightarrow Q$ are useful in certain applications; we refer the reader to \cite{absil2007optimization}. Thus it can be desirable consider a more general scheme in which $F_i$ factors as $E \circ f_i$ for a $C^1$-smooth map $E : TQ \rightarrow Q$.

On the other hand, one would sometimes like to consider the case in which $F_i$ factors through a Lie algebra. In particular, suppose that $F_i$ factors as
\begin{equation}
\label{eq:lie-algebra-factorization}
  \begin{tikzcd}
    Q \times U \arrow[rr,"F_i"] \arrow{dr}[swap]{f_i \times \pi_1} & & Q \\
    & \g \times Q \arrow{ur}[swap]{\lambda}
  \end{tikzcd}
\end{equation}
where $\pi_1$ denotes projection onto the first coordinate and $\lambda : \g \times Q \rightarrow Q$ is a left action of a Lie algebra $\g$ on $Q$. For example, the update scheme
\begin{equation*}
  q_i \colonequals \lambda\left( h_{i-1}f(q_{i-1}, u_{i-1}), q_{i-1}\right)
\end{equation*}
arises naturally in the \emph{Runge-Kutta-Munthe-Kaas} method for numerical solution of ODEs on manifolds \cite{munthe1999high}.
We recommend \cite{iserles2000lie} for a comprehensive overview of this and related techniques.

With the above examples in mind, we consider in this section a vector bundle $\pi : X \rightarrow Q$ along with a $C^1$-smooth map $E : X \rightarrow Q$. We suppose that the update maps $F_i$ factor as
\begin{equation}
\label{eq:factorization}
  \begin{tikzcd}
    Q \times U \arrow[rr,"F_i"] \arrow{dr}[swap]{f_i} & & Q \\
    & X \arrow{ur}[swap]{E}
  \end{tikzcd}
\end{equation}
where $f_i : Q \times U \rightarrow X$ are maps satisfying $f_i(q,u) \in X_q \colonequals \pi^{-1}(q)$. This scheme includes \eqref{eq:riemannian-factorization}, \eqref{eq:lie-algebra-factorization}, and others.

Let us recall the following:
\begin{definition}
For a vector bundle $\pi : X \rightarrow Q$ and a smooth map $E : X \rightarrow M$, the \emph{fibre derivative} of $E$ at $v \in X$ is the linear map $\mathbb{F}E_v : X_{\pi(v)} \rightarrow T_{E(v)}M$ defined by
\begin{equation*}
\mathbb{F}E_v(w) \colonequals \left. \frac{d}{dt} \right|_{t = 0} E(v + t w).
\end{equation*}
\end{definition}
We introduce the following assumptions:

\textbf{Assumptions ($\mathsf{A}$):} The maps $F_i$ factor as \eqref{eq:factorization}; the control manifolds $U_i$ are subsets of $\R^{k_i}$ for natural numbers $k_0, \dots, k_{n-1}$; control sets $\U_i$ are closed and convex; mappings $f_i : Q \times \U_i \rightarrow X$ are affine in $u$ for each fixed $q$; and functions $L_i : Q \times \U_i \rightarrow \R$ are convex in $u$ for each fixed $u$.

\begin{theorem}
\label{thm:discrete-geometric-pmp}
  Suppose that control $\s{u}$ is a local minimizer for the cost function \eqref{eq:jdef} and that in addition to the standing assumptions, assumptions \textnormal{\textbf{(}}$\mathsf{A}$\textnormal{\textbf{)}} are true.
Then there are sequences $(a_i)_{i = 1}^n$ and $(b_i)_{i = 0}^{n-1}$ which satisfy $b_0 \in \pL L_0(q_0, u_0)$, $(a_i, b_i) \in \pL L_i(q_i, u_i)$ for $1 \le i \le n-1$, and $a_n \in \pL \ell(q_n)$ such that the sequence of costates defined by $p_n \colonequals - a_n$ and \eqref{eq:discrete-adjoint} satisfies the maximum principle
  \begin{equation}
  \label{eq:max-princ-fiber-version}
  H_i(u_i, p_{i+1}) = \max_{u \in \U_i} H_i(u, p_{i+1}) \hspace{10pt} (0 \le i \le n-1)
  \end{equation}
where $H_i : \U_i \times T_{q_{i+1}}^* Q \rightarrow \R$ is defined through
\begin{equation*}
  H_i(u,p) \colonequals \left<\mathbb{F}E_{f_i(q_i,u_i)}^* p, f(q_i,u) \right> - L_i(q_i,u).
\end{equation*}
\end{theorem}

\begin{proof}
Since $\s{u}$ is a local minimizer for $J$, we have $\underline{D}J(\s{u};\s{v}) \ge 0$ for all $\s{v} \in T_{\cU}(\s{u})$. Applying Theorem \ref{thm:dmp} with $\Delta = 0$ we obtain everything except \eqref{eq:max-princ-fiber-version}. To prove \eqref{eq:max-princ-fiber-version} we first check that because $f$ is control affine
 \begin{align*}
D_u F_i(q,u) \left(v - u\right) &= \left. \frac{d}{dt} \right|_{t = 0} E \left( (1-t) f(q, u) + t f(q,v) \right)\\& = \mathbb{F}_{f(q,u)}E \left(f(q,v) - f(q,u)\right).
 \end{align*}

We next recall \cite[Proposition 2.5.5]{clarkeand1998} that for a closed convex set $\U_i \subset \R^k$
\begin{equation*}
  T_{\U_i}^C(u_i) = c \ell \bigcup_{t > 0} t \left(\U_i - u_i\right).
\end{equation*}
Thus inequality \eqref{eq:discrete-amp} implies
\begin{equation}
\label{eq:sufficient-condition}
0 \le  \left<b_i, u - u_i \right> - \left< \mathbb{F}_{f_i(q_i,u_i)}E^* p_{i+1}, f(q_i,u) - f(q_i,u_i) \right>
\end{equation}
Since $\U_i$ is a convex set and $L$ is convex in control, \eqref{eq:sufficient-condition} is sufficient for \eqref{eq:max-princ-fiber-version}.
\end{proof}

\subsection{Lie Groups and Variational Integrators}
\label{sec:liegroups}

We now apply Theorem \ref{thm:dmp} to the study of control problems on Lie groups by allowing controls to be group elements and using multiplication for the update maps. More precisely, suppose that we are given a Lie group $G$ with Lie algebra $\g$, let the control sets $\U$ be given by $\U = U = G$, and consider the control system
\begin{equation}
\label{eq:multiplication-lie}
  g_{i} \colonequals g_{i-1} u_{i-1}.
\end{equation}

Let us agree to identify $TG$ with $G \times \g$ through $T_gG = \left\{ DL_{g}a \, : \, a \in \g\right\}$, where $L_g : G \rightarrow G$ is the left multiplication map $L_g(h) \colonequals gh$.
Let $\mu : G \times G \rightarrow G$ denote the multiplication map and $\psi : G \rightarrow \mathrm{Aut}(G)$ the automorphism map $\psi(g)(h) \colonequals ghg^{-1}$.
Recall that the \emph{adjoint representation} of $G$ is the map $\mathrm{Ad} : G \rightarrow \mathrm{Aut}(\g)$ given by $\mathrm{Ad}(g) \colonequals D\psi(g)(e)$ and the \emph{coadjoint representation}, which is the map $\mathrm{Ad}^* : G \rightarrow \mathrm{Aut}(\g^*)$ given by $\mathrm{Ad}^*(g) \colonequals \mathrm{Ad}(g^{-1})^*$.

For ease of notation, we will write $\mathrm{Ad}^*_g\colonequals\mathrm{Ad}^*(g)$.

\begin{theorem}
\label{thm:a-lie-group-theorem}
If $\s{u}$ is critical for \eqref{eq:jdef} with $F_i$ given by \eqref{eq:multiplication-lie}, $\ell : G \rightarrow \R$ locally Lipschitz, and $L_i : G \times G \rightarrow \R$ $C^1$-smooth, then the sequence $p_i \in \g^*$ defined by
\begin{equation}
  \label{eq:definition-Lie-pi}
p_i = d_u L_{i-1}(g_{i-1}, u_{i-1}) \hspace{10pt} (1 \le i \le n)
\end{equation}
evolves according to
\begin{equation}
  \label{eq:Lie-adjoint}
p_{i+1} = \mathrm{Ad}^*_{u_i}( p_i) + \mathrm{Ad}^*_{u_i}( d_gL_i(g_i,u_i)) \hspace{10pt} (1 \le i \le n-1)
\end{equation}
and satisfies the endpoint condition
\begin{equation}
  \label{eq:Lie-endpoint}
  -p_n \in \pL \ell(g_n).
\end{equation}
\end{theorem}

\begin{proof}
In order to apply Theorem \ref{thm:dmp} we remind the reader of the following lemma:
\begin{lemma}
\label{lem:derivatives}
  As maps from $\g$ to $\g$, $D_1 \mu(g,h) = \mathrm{Ad}_{h^{-1}}$ and $D_2 \mu(g,h) = Id_{\g}$.
\end{lemma}
\begin{proof}
  Note that $\left. \frac{d}{dt}\right|_{t = 0} g \exp(ta)h = \left. \frac{d}{dt}\right|_{t = 0} gh h^{-1} \exp(ta) h = DL_{gh} \mathrm{Ad}_{h^{-1}}(a)$. The proof for $Id_{\g}$ is similar.
\end{proof}

Thus from Theorem \ref{thm:dmp} and Lemma \ref{lem:derivatives} we obtain a sequence $(p_i)_{i = 1}^n$ which satisfies \eqref{eq:Lie-endpoint} and
\begin{equation}
\label{eq:abstract-lie-adjoint}
  p_i = -d_g L_i(g_i,u_i) + \mathrm{Ad}_{u^{-1}_i}^*(p_{i+1}),
\end{equation}
Solving \eqref{eq:abstract-lie-adjoint} for $p_{i+1}$ we obtain \eqref{eq:Lie-adjoint}. That $p_i$ satisfies \eqref{eq:definition-Lie-pi} follows from \eqref{eq:discrete-amp}, which holds for all $v \in T_{\U}^C(u_i) = T_{u_i}G$.
\end{proof}

We point out that \eqref{eq:definition-Lie-pi} corresponds to the discrete-time Legendre transform
\begin{equation}
\label{eq:discrete-legendre}
p = \mathbb{F}^+L(g,u) \colonequals d_u L_i(g,u)
\end{equation}
defined in \cite{marsden2001}.
Indeed there is a deeper connection to the discrete-time mechanics of \cite{marsden2001}. Let us consider the particular case in which $J$ is a function of the type
\begin{equation}
\label{eq:action-sum}
  J(\s{u}) = \sum_{i = 0}^{n-1} \left\{h K(u_i) - \frac{h}{2} \varphi(g_i) - \frac{h}{2} \varphi(g_iu_i) \right\},
\end{equation}
where $K$ represents a discrete-time kinetic energy and $\varphi$ a potential.\footnote{Strictly speaking, such physical notions are not necessary for Theorem \ref{thm:variational-stepping}. However in most applications such functions correspond to discrete-time action sums (see, for example, \cite{lee2005lie}).}

\begin{theorem}
\label{thm:variational-stepping}
Suppose that $\s{u}$ is critical for a function of the form \eqref{eq:action-sum}, functions $K$ and $\varphi$ are $C^2$-smooth, and $K$ satisfies:
\begin{enumerate}[(i)]
\item As a smooth map between manifolds, $dK : G \rightarrow \g^*$ is full rank at $e \in G$;
\item $dK(e) = 0$.
\end{enumerate}
Let $p_i$ denote the discrete-time momentum $p_i \colonequals \mathbb{F}^+ L(g_i, u_i)$ and let $m_i \colonequals d\varphi(g_i) \in \g^*$.
Then for $1 \le i \le n-1$ and sufficiently small step size $h$, $(g_{i+1}, p_{i+1})$ may be obtained from $(g_i,p_i)$ by solving, first for $u_i$, then $g_{i+1}$, and finally $p_{i+1}$, the equations
\begin{align}
p_i &= h \left( \mathrm{Ad}^*_{u^{-1}_i} dK(u_i) +\frac{h}{2} m_i \right) \label{eq:moser-veselov}
 \\
g_{i+1} & = g_iu_i \nonumber \\
\label{eq:lie-adjoint} p_{i+1} &= \mathrm{Ad}^*_{u_i}( p_i) - \frac{h}{2}\mathrm{Ad}^*_{u_i}( m_i) - \frac{h}{2}m_{i+1} .
\end{align}
\end{theorem}

\begin{proof}
   From Theorem \ref{thm:a-lie-group-theorem} we obtain a sequence $(p_i)_{i = 1}^n$ which evolves according to \eqref{eq:lie-adjoint}. We also obtain from \eqref{eq:definition-Lie-pi} the equation
\begin{equation*}
  p_{i+1} = h dK(u_i) - \frac{h}{2} m_{i+1}
\end{equation*}
and by plugging this into \eqref{eq:lie-adjoint} and rearranging we obtain \eqref{eq:moser-veselov}. We need only check that \eqref{eq:moser-veselov} is solvable for sufficiently small $h$.

To see this, note that by our first assumption on $K$, $dK$ restricts to a diffeomorphism on a neighborhood of $e \in G$. The image of this neighborhood includes $0$ by the second assumption and so equation \eqref{eq:moser-veselov} is uniquely solvable for sufficiently small step size $h$.
\end{proof}

We remark that for $G = SO(3)$ with
\begin{equation*}
  K(u) = \frac{1}{h} \mathrm{tr} \left( \left(I_{3 \times 3}  - u\right)J_d \right)
\end{equation*}
for $J_d$ a positive definite matrix, this scheme corresponds to the Hamiltonian equations derived in \cite{lee2005lie}. In this case \eqref{eq:moser-veselov} reduces to the famous equation of Moser and Veselov \cite{moser1991}.

\section{Exact Penalization and Sensitivity}
\label{sec:decrease-princ}

For the remainder of the paper we undertake a careful study of necessary optimality conditions for problems of discrete-time geometric optimal control. As before we consider a cost function $J : \cU \rightarrow \R$ defined by \eqref{eq:jdef} and we now suppose that the controls and the states are subject either to pure state constraints of the type $q_i \in S_i \subset Q$ or to mixed constraints of the type $(q_i, u_i) \in S_i \subset Q \times U_i$. We introduce in Section \ref{sec:mixed-constraints} a constraint qualification which corresponds to abnormality of necessary optimality conditions and we derive in that section results on sensitivity of the value function to perturbations in constraints. In Section \ref{sec:pmp-constraints} we present a discrete-time geometric maximum principle.

These results are rooted in Clarke's technique of exact penalization and in this section we prove a lemma which provides the link between abnormality of constraints and exact penalization.
Since we will make extensive use of the exact penalization technique in the remainder of this paper, we present it here and refer the reader to \cite{clarkeand1998} for further applications.

Let us recall that if $\cX$ is a metric space with metric $d_\cX$ and $S \subset \cX$ is a set then
\begin{equation*}
  d_S(x) \colonequals \inf \left\{ d_{\cX}(x,s) \, : \, s \in S\right\}.
\end{equation*}

\begin{lemma}[Exact Penalization]
Let $S \subset X$ be closed and $f : \cX \rightarrow \R$ locally Lipschitz with constant $K_f$. For any $x_* \in S$ and $\ve \ge 0$ satisfying $f(x_*) \le \inf \left\{ f(s) \, :\, s \in S\right\} + \ve$ there exists $\delta > 0$ such that for all $K > K_f$ we have
\begin{equation*}
\label{eq:exact-penalization}
  f(x_*) \le \inf \left\{ f(x) + K d_S(x)\, : \, d_{\cX}(x,x_*) \le \delta \right\}+\ve.
\end{equation*}
\end{lemma}
\begin{proof}
Let $\delta > 0$ be such that for all $x_1,x_2 \in \cX$ satisfying $d_\cX(x_i,x_*) \le 3\delta$ we have $f(x_1) - f(x_2) \le K_f d_\cX(x_1,x_2)$ and let $x \in \cX$ be such that $d_\cX(x,x_*) \le \delta$. Note that if $x \in S$ then we have $f(x) \ge f(x_*) - \ve$.

If $x \not \in S$ then for any $0 < \gamma < \delta$ we may pick $s \in S$ satisfying $d_\cX(s,x) \le d_S(x) + \gamma$. The reader can check that $d_\cX(s,x_*) < 3 \delta$ and so we can write
\begin{align*}
  f(x) + Kd_S(x) & \ge f(s) - K_f d(x,s) + K d_S(x) \\
  & \ge f(x_*) - \ve - K_f d_S(x) + K d_S(x) - K_f \gamma \\
  & \ge f(x_*) - \ve - K_f \gamma.
\end{align*}
Letting $\gamma \downarrow 0$ we see that $f(x) + K d_S(x) \ge f(x_*) - \ve$ for all $x$ satisfying $d_\cX(x,x_*)\le \delta$.
\end{proof}

Thus a local minimizer or $\ve$-minimizer of $f$ over the set $S$ is a local unconstrained minimizer (or unconstrained $\ve$-minimizer) of the function $f + Kd_S$.

In the following we will also need the next result on Clarke tangent cone:
\begin{lemma}
  \label{lem:clarke-tangent}
  Let $S$ be closed and $v \in T_S^C(q)$. For any $q_n \rightarrow q$ and $t_n \downarrow 0$ and any chart $\varphi : \cO \rightarrow \R^d$ whose domain includes $q$ there exists a sequence $r_n \rightarrow q$ for which $r_n \in S$ and
  \begin{equation}
  \label{eq:sequence-characterization-tangent}
    \lim_{n \rightarrow \infty} \frac{\varphi(r_n) - \varphi(q_n)}{t_n} = \varphi_*(q) v.
  \end{equation}
\end{lemma}
\begin{proof}
Let $x \colonequals \varphi(q)$ and $x_n \colonequals \varphi(q_n)$. By \eqref{eq:local-tangent} we have $\varphi_*(q)v \in T_{\varphi(\cO \cap S)}^C(x)$ and so we can use \cite[Proposition 2.5.2]{clarkeand1998}
to obtain a sequence $w_n \rightarrow \varphi_*(q)v$ for which $x_n + t_n w_n \in \varphi(\cO \cap S)$. Setting $r_n \colonequals \varphi^{-1}(x_n + t_n w_n)$ gives us \eqref{eq:sequence-characterization-tangent}.
\end{proof}

We mention the following specialization, which can be useful in the case of Riemannian manifolds:

\begin{lemma}
  \label{lem:clarke-tangent-riemannian}
  Let $S \subset M$ be a closed subset of a Riemannian manifold and $v \in T_S^C(q)$. Then for any sequence $t_n \downarrow 0$ there exists a sequence $v_n \in T_qM$ for which $v_n \rightarrow v$ and $\exp(t_n v_n) \in S$.
\end{lemma}

\begin{proof}
In Lemma \ref{lem:clarke-tangent} we take $q_n \equiv q$ and $\varphi^{-1} \colonequals \exp$, so that $\varphi^{-1} : T_qM \cong \R^d \rightarrow M$. Repeating the same proof as above we obtain a sequence $v_n \in T_qM$ for which $\varphi^{-1}(t_nv_n) \in S$.
\end{proof}

\subsection{Decrease Condition}

Let us consider a Riemannian manifold $M$ with metric $g$ and distance function $d_g$. We suppose we are given a closed subset $\cU \subset M$, a metric space $E$ with metric $d_E$, and a nonnegative function $P : E \times M \rightarrow \R$. We are interested in the set $\cA(e) \subset \cU$ consisting of those $q \in \cU$ for which $P(e,q) = 0$.

We make the following definition:
\begin{definition}
\label{defn:strong-decrease}
  Function $P$ satisfies the \emph{strong decrease condition} for $\cU$ near $q_0 \in \cA(e_0)$ if there exist $\ve, \Delta > 0$ such that for any $(e,q) \in E \times \cU$ satisfying $P(e,q) > 0$ and
  \begin{equation*}
  d_g(q,q_0) + d_E(e,e_0) < \ve
\end{equation*}
there exists nonzero $v \in T_{\cU}^C(q)$ such that
\begin{equation*}
  \liminf_{\lambda \downarrow 0} \frac{P(e, c_v(\lambda)) - P(e,q)}{\lambda} \le - \Delta \left\|v\right\|_g,
\end{equation*}
where $c_v : \R \rightarrow M$ is any smooth curve satisfying $c_v^\prime(0) = v$.
\end{definition}

We remark that because the lower Dini derivative of Lipschitz functions is positive homogeneous in $v$, it suffices to suppose that $\left\|v\right\|_g = 1$ when $P$ is Lipschitz in $q$. Moreover, we are requiring in Definition \ref{defn:strong-decrease} only that $c_v : \R \rightarrow M$, not $c_v : \R \rightarrow \cU$. The Lemmas \ref{lem:clarke-tangent} and \ref{lem:clarke-tangent-riemannian} ensure that this is enough, as we will see below.

\begin{theorem}
\label{thm:penalty}
If $P$ is jointly continuous in $(e,q)$; Lipschtiz in $q$ for $(e,q)$ in a neighborhood of $(e_0,q_0)$; and satisfies the strong decrease condition near $q_0 \in \cA(e_0)$ then there exists a neighborhood $\cO\subset E \times \cU$ of $(e_0,q_0)$ such that for all $(e,q) \in \cO$ there holds
\begin{equation}
\label{eq:penalty}
  d_{\cA(e)}(q) \le \frac{1}{\Delta} P(e,q).
\end{equation}
In particular, the sets $\cA(e)$ are nonempty.
\end{theorem}

\begin{proof}
  Suppose that $P$ satisfies the strong decrease condition near $q_0 \in \cA(e_0)$ and fix $\ve, \Delta > 0$ as in Definition \ref{defn:strong-decrease}. Since $P(e_0,q_0) = 0$ we may choose $0 < \ve_* < \frac12 \ve$ such that for all $(e,q) \in E \times \cU$ satisfying
  \begin{equation}
  \label{eq:distance-less-than-ve-star}
      d_g(q,q_0) + d_E(e,e_0) < \ve_*
  \end{equation}
we have
 \begin{equation}
  \label{eq:P-is-small}
    \frac{4}{\Delta} P(e,q) < \frac12 \ve.
  \end{equation}

Now suppose by way of contradiction that there exists $(e_1,q_1) \in E \times \cU$ which satisfies \eqref{eq:distance-less-than-ve-star} but not \eqref{eq:penalty}. Then we may fix $1 < \gamma < 2$ such that
  \begin{equation}
  \label{eq:assume-bwoc}
      d_{\cA(e_1)}(q_1) > \frac{\gamma^2}{\Delta} P(e_1,q_1).
  \end{equation}
Inequality \eqref{eq:assume-bwoc} implies $P(e_1,q_1) > 0$. Consequently, because $\gamma$ is strictly greater than $1$ and $P$ is nonnegative we must have
 \begin{equation*}
   P(e_1,q_1) < \inf \left\{ P(e_1,q) \, : \, q \in \cU \;\; \mathrm{s.t.} \;\; d_g(q_0,q) \le 2\ve \right\} + \gamma P(e_1,q_1).
 \end{equation*}
 Applying the Ekeland variational principle to the function $q \mapsto P(e_1,q)$ defined on the complete metric space $\left\{q \in \cU \, : \, d_g(q_0,q) \le 2\ve \right\}$ with
 \begin{equation*}
   \ve = \gamma P(e_1,q_1), \, \sigma = \frac{\gamma^2}{\Delta} P(e_1,q_1)
 \end{equation*}
we obtain $q_2 \in \cU$ with
  \begin{equation}
  \label{eq:ekeland-sigma}
    d_g(q_2,q_1) < \frac{\gamma^2}{\Delta} P(e_1,q_1)
  \end{equation}
  which minimizes the function
  \begin{equation}
  \label{eq:ekeland-penalty-function}
    q \mapsto P(e_1,q) + \frac{\Delta}{\gamma} d_g(q_2,q)
  \end{equation}
  over $q \in \cU$ satisfying $d_g(q_0,q) \le 2 \ve$.

Combining inequalities \eqref{eq:assume-bwoc} and \eqref{eq:ekeland-sigma} we find
$d_g(q_2, q_1) < d_{\cA(e_1)}(q_1)$ and so $q_2 \not \in \cA(e_1)$. It follows that $P(e_1,q_2) > 0$. In addition, from \eqref{eq:distance-less-than-ve-star} and \eqref{eq:ekeland-sigma} we see
  \begin{equation*}
    d_g(q_2,q_0) + d_E(e_1,e_0) \le d_g(q_2, q_1) + d_g(q_1, q_0)+ d_E(e_1,e_0) \le \ve_* +\frac{4 }{\Delta} P(e_1,q_1).
  \end{equation*}
Thus \eqref{eq:P-is-small} implies that
$d_g(q_2,q_0) + d_E(e_1,e_0)  \le \frac12 \ve + \frac12 \ve = \ve$ and so there exists a nonzero $v \in T_{\cU}^C(q_2)$ and a smooth curve $c_v : \R\rightarrow M$ satisfying $c_v^\prime(0) = v$ and
\begin{equation*}
  \liminf_{\lambda \downarrow 0} \frac{P(e_1,c_v(\lambda)) - P(e_1, q_2)}{\lambda} \le - \Delta \left\|v\right\|_g.
\end{equation*}

We now pick a sequence $(\lambda_n)_{n = 1}^\infty$ satisfying
\begin{equation}
\label{eq:strong-decrease-e1q2}
  \lim_{n \rightarrow \infty} \frac{P(e_1,c_v(\lambda_n)) - P(e_1, q_2)}{\lambda_n} \le - \Delta \left\|v\right\|_g
\end{equation}
and apply Lemma \ref{lem:clarke-tangent-riemannian} to obtain a sequence of vectors $v_n \in T_{q_2}M$ with $v_n \rightarrow v$ and $\exp_{q_2}(\lambda_n v_n) \in \cU$. Since the map $q \mapsto P(e,q)$ is locally Lipschitz, \eqref{eq:strong-decrease-e1q2} implies
\begin{equation*}
  \lim_{n \rightarrow \infty} \frac{P(e_1,\exp_{q_2}(\lambda_n v_n)) - P(e_1, q_2)}{\lambda_n} \le - \Delta \left\|v\right\|_g.
\end{equation*}

At the same time, because $q_2$ is a local minimizer for the function \eqref{eq:ekeland-penalty-function} over $\cU$ and because $\exp_{q_2}(\lambda_n v_n) \in \cU$ we have
\begin{equation}
\label{eq:optimality}
  P(e_1, q_2) \le P(e_1, \exp_{q_2}(\lambda_n v_n)) + \frac{\Delta}{\gamma} d_g(q_2, \exp_{q_2}(\lambda_n v_n))
\end{equation}
Now use $d_g(q_2, \exp_{q_2}(\lambda_n v_n)) = \lambda_n \left\|v_n\right\|_g$ and \eqref{eq:optimality} to obtain, for sufficiently large $n$,
\begin{equation*}
  0 \le \frac{ P(e_1, \exp_{q_2}(\lambda_n v_n)) -   P(e_1, q_2)}{\lambda_n} + \frac{\Delta}{\gamma} \left\|v_n\right\|_g.
\end{equation*}

Letting $n \rightarrow \infty$ we arrive at the contradiction
\begin{equation*}
0 \le -\Delta \left\|v\right\|_g+ \frac{\Delta}{\gamma} \left\|v\right\|_g < 0,
\end{equation*}
since $\gamma > 1$. It must therefore be that \eqref{eq:penalty} holds for all $(e,q) \in E \times \cU$ satisfying \eqref{eq:distance-less-than-ve-star}.
\end{proof}

\subsection{Exact Penalization for Regular Constraints}
\label{sec:mixed-constraints}

We now apply Theorem \ref{thm:penalty} to an derive exact penalization result for constraints of the type
 \begin{equation*}
   x \in S(e) \subset M,
 \end{equation*}
 where $M$ is a Riemannian manifold and $S$ is a family of closed sets depending on a parameter $e$ in a metric space $E$. Because we are working on a Riemannian manifold it will be useful to recall the following facts:
 \begin{lemma}
\label{lem:norm-lemma}
If $a \in \pL d_S(q)$ then $\left\|a\right\|_g \le 1$. If $q \not \in S$ then $\left\|a\right\|_g = 1$ and if $q \in S$ then $a \in N_S^L(q)$.
\end{lemma}
\begin{proof}
  The facts follow from Proposition 3.5 and Corollary 7.14 in \cite{ledyaevzhu2007}.
\end{proof}

\subsubsection{Regularity of Constraints}
We assume we are given a nonnegative function $\varphi : E \times M \rightarrow \R$, locally Lipschitz in $x$ uniformly with respect to $e$, which characterizes $S(e)$ in the sense that $S(e) = \left\{x \, : \, \varphi(e,x) = 0\right\}$. This assumption can be made without loss of generality as one may always take $\varphi(e,x) \colonequals d_{S(e)}(x)$.

\begin{definition}
\label{defn:regular-constraints}
  We say that $S(e_0) \subset M$ is $\varphi$-\emph{regular} at $x_0 \in S(e_0)$ if:
\begin{enumerate}[(i)]
\item There exists $\ve, \Delta > 0$ such for any $(e,x)$ satisfying $d_E(e,e_0) + d_M(x,x_0) < \ve$ and $\varphi(e,x) > 0$ and any $p \in \partial_{L,x}\varphi(e,x)$ we have $\left\|p\right\| \ge \Delta$;
\item If $(e_i, x_i) \rightarrow (e_0, x_0)$ and $p_i \in \partial_{L,x} \varphi(e_i,x_i)$ then there is a subsequence for which $p_{i_j} \rightarrow p \in \partial_{L,x} \varphi(e_0,x_0)$.
\end{enumerate}
\end{definition}

This type of condition is quite common (see, for example, \cite[Section 3.3]{clarkeand1998} or \cite[Theorem 3.6.3]{borweinzhu2005}) and covers a very broad range of applications.

For example, if a closed set $S \subset M$ has no parameter dependence then $S$ is $d_S$-regular at each $x \in S$. Indeed in this case $(ii)$ reduces to boundedness of the sequence $(p_i)_{i = 1}^\infty$ and definition of limiting subgradient. In $(i)$ we may take $\Delta = 1$ by Lemma \ref{lem:norm-lemma}. We emphasize that if one is willing to work with the distance function then \emph{every closed set $S$ is $d_S$-regular} in the sense of Definition \ref{defn:regular-constraints}.

Distance to a closed set can be difficult to calculate however, even in $\R^d$, and so we point out a second important example covered by $\varphi$-regularity. Suppose that we have $r$ inequality constraints $g_i(x) \le e_i$ and $s$ equality constraints $h_j(x) = e_{r+j}$, with $e \in \R^{r+s}$ and functions $g_i, h_j$ $C^1$-smooth. Let
 \begin{equation}
 \label{eq:max-type-penalty-function}
   \varphi(e,x) \colonequals \max\left\{0,g_1(x) - e_1, \dots, g_r(x)- e_r, \left|h_1(x) - e_{r+1}\right|, \dots, \left|h_s(x) -  e_{r+s}\right| \right\}
 \end{equation}
 so that
 \begin{equation*}
   S(e) = \left\{ x \in M \, : \, g_i(x) \le e_i\, , \, h_j(x) = e_{r + j} \right\}.
 \end{equation*}

 \begin{lemma}
 \label{lem:equality-inequality-regular}
If the only solution to
\begin{equation*}
  \sum_{i = 1}^r \lambda_i dg_i(x_0) + \sum_{j = 1}^s \mu_j dh_j(x_0) = 0
\end{equation*}
for $(\lambda, \mu) \in \R^{r+s}$ satisfying the nonnegativity condition $\lambda_i \ge 0$ and complementary slackness condition $\lambda_i g_i(x) = 0$ is the trivial solution then $S(0)$ is $\varphi$-regular at $x_0$.
 \end{lemma}

 \begin{proof}
We refer the reader to the Lemma of page 132 in \cite{clarkeand1998}, along with the formula for the subgradient of a maximum-type function applied in local coordinates. For the latter, see \cite[Theorem 3.5.8]{borweinzhu2005}. We remark that while the result in \cite{clarkeand1998} is given in terms of the proximal subgradient, the same argument goes through without change for Fr\'echet subgradient.
 \end{proof}

In problems of discrete-time control we will suppose that states $q_i$ and controls $u_i$ are subject, in addition to the usual control constraint $u_i \in \U_i \subset U_i$, to mixed constraints of the form $(q_i, u_i) \in S_i \subseteq Q \times U_i$ with $S_i$ depending on a parameter $e$ in a metric space $E$. We thus require:
\begin{equation*}
 u_i  \in \U_i \hspace{10pt} (0 \le i \le n-1)
 \end{equation*}
and
\begin{equation}
\label{eq:abstract-mixed-constraints}
\left\{\begin{aligned}
u_0 & \in S_0(e) \subset U_0 \\
(q_i,u_i) & \in S_i(e) \subset Q \times U_i \hspace{10pt} (1 \le i \le n-1) \\
q_n & \in S_n(e) \subset Q. \end{aligned}\right.
\end{equation}
 As before we let $\cU \subset \prod_{i = 0}^{n-1} U_i$ denote the set of sequences $\s{u}$ for which $u_i \in \U_i$. We will write $\cA(e) \subset \cU$ for the set of all control sequences $\s{u}$ for which the corresponding state sequence $\s{q}$ satisfies, along with with $\s{u}$, constraints \eqref{eq:abstract-mixed-constraints}.

 We will require these constraints to be $\varphi$-regular in the sense of Definition \ref{defn:regular-constraints} at $u_0 \in S_0(e)$, $(q_1, u_1) \in S_1(e)$, and so on. In order to avoid a lengthy list of assumptions with each theorem we make the following definition:
 \begin{definition}
   \label{defn:regular-along}
   Constraints \eqref{eq:abstract-mixed-constraints} are $\varphi_i$-\emph{regular along} $\s{u} \in \cA(e)$ if $S_0(e)$ is $\varphi_0$-regular at $u_0$, $S_i(e)$ is $\varphi_i$-regular at $(q_i,u_i)$ for $1 \le i \le n-1$, and $S_n(e)$ is $\varphi_n$-regular at $q_n$.
 \end{definition}

We assume a different type of regularity on the control sets:
\begin{definition}
\label{defn:globally-regular}
  A closed set $S \subset M$ is \emph{Clarke regular} if for every $s \in S$, $v \in T_S^C(s)$, and $s_n \rightarrow s$ there exists a sequence $v_n \in T_S^C(s_n)$ which converges to $v$.
\end{definition}
Closed, convex subsets of $\R^m$ are Clarke regular and the condition can be characterized entirely in terms of the \emph{Bouligand tangent cone} (see \cite[Corollary 3.6.13]{clarkeand1998}).

\subsubsection{Penalty Functions and Constraint Qualification}

Let functions $\varphi_i : E \times M \rightarrow \R$ characterize constraints \eqref{eq:abstract-mixed-constraints} and consider a function $P : E \times \cU \rightarrow \R$ defined through
\begin{equation}
\label{eq:pdefn}
  P(e,\s{u}) \colonequals \varphi_0(e,u_0) + \varphi_n(e,q_n) + \sum_{i = 1}^{n-1} \varphi_i(e, q_i, u_i),
\end{equation}
so that $\s{u} \in \cA(e)$ if and only if $\s{u} \in \cU$ and $P(e,\s{u}) = 0$. Let us suppose that we are given some $\ol{\s{u}} \in \cA(\ol{e})$. The results of this section provide a sufficient condition in the form of a constraint qualification for the following inequality to hold for $(e,\s{u})$ sufficiently close to $(\ol{e}, \ol{\s{u}})$ and $\kappa$ sufficiently large:
\begin{equation*}
  d_{\cA(e)}(\s{u}) \le \kappa P(e,\s{u}).
\end{equation*}
The sufficient condition we provide is given in terms of the following constraint qualification:
\begin{definition}
\label{defn:strictly-normal}
  Control $\s{u} \in \cA(e)$ is said to be \emph{strictly $\varphi_i$-normal} if the only sequences
   \begin{equation*}
\left\{ \begin{aligned} b_0 &\in \partial_{L,u} \varphi_0(e,u_0) \\
(a_i, b_i) &\in \partial_{L,(q,u)} \varphi_i(e,q_i,u_i) \hspace{10pt} (1 \le i \le n-1) \\
a_n &\in \partial_{L,q}\varphi_n(q_n) \end{aligned} \right.
   \end{equation*}
which generate a costate sequence $(p_i)_{i = 1}^n$ through $-p_n = a_n$ and
  \begin{equation}
  \label{eq:normal-adjoint}
p_{i-1} = a_{i-1} - D_q F_{i-1}(q_{i-1}, u_{i-1})^*p_i
  \end{equation}
that satisfies for all $v \in T_{\U_i}^C(u_i)$
  \begin{equation}
  \label{eq:normal-max}
0 \le \left<b_{i-1} - D_u F_{i-1}(q_{i-1}, u_{i-1})^* p_i, v \right> \hspace{10pt} (1 \le i \le n-1)
  \end{equation}
  are the sequences $\s{a} \equiv 0$ and $\s{b} \equiv 0$.
\end{definition}

\begin{theorem}
\label{thm:strong-decrease-necessary}
   If the sets $\U_i$ are Clarke regular, constraints \eqref{eq:abstract-mixed-constraints} $\varphi_i$-regular along $\s{u}$, and control $\ol{\s{u}} \in \cA(\ol{e})$ strictly $\varphi_i$-normal then the function $P$ defined by \eqref{eq:pdefn} satisfies the strong decrease condition for $\cU$ near $(\ol{e}, \ol{\s{u}})$.
\end{theorem}

\begin{proof}
Let $\delta > 0$ be the minimum of the numbers $\Delta$ appearing in Definition \ref{defn:regular-constraints}. Thus, for example, if $(e,q,u)$ is sufficiently close to $(\ol{e}, \ol{q}_i, \ol{u}_i)$, $\varphi_i(e,q,u) > 0$, and $(a,b) \in \partial_{L,(q,u)} \varphi_i(e,q,u)$ then $\left\|(a,b)\right\| \ge \delta$.

Suppose by way of contradiction that there exists a sequence $(\Delta_k)_{k = 1}^\infty$ with $\Delta_k \downarrow 0$ such that for each $k$ there exist $e_k \in E$ and $\s{u}_k \in \cU$ with the properties $d_E(e_k,\ol{e}) + d(\s{u}_k, \ol{\s{u}}) < \Delta_k$, $P(e_k,\s{u}_k) > 0$, and for any nonzero $\s{v} \in T_{\cU}^C(\s{u}_k)$
  \begin{equation*}
\liminf_{\lambda \downarrow 0} \frac{P(e_k, c_{\s{v}}(\lambda)) - P(e_k, \s{u}_k)}{\lambda} \ge - \Delta_k \left\|\s{v}\right\|_g,
  \end{equation*}
  where $c_{\s{v}} : \R \rightarrow \prod_{i = 0}^{n-1} U_i$ is a smooth function satisfying $c_{\s{v}}^\prime(0) = \s{v}$. Let $\s{q}_k$ be the trajectory for control $\s{u}_k$.

From Theorem \ref{thm:dmp} we obtain sequences $(\s{a}_k)_{k = 1}^\infty$, $(\s{b}_k)_{k = 1}^\infty$ and $(\s{p}_k)_{k = 1}^\infty$ with $b_{k,0} \in \partial_{L,u} \varphi_0(e_k, u_{k,0})$,  $(a_{k,i},b_{k,i}) \in \partial_{L,(q,u)}\varphi_i(e_k, q_{k,i},u_{k,i})$, $-p_{k,n} = a_{k,n} \in \partial_{L,q} \varphi_n(e_k,q_{k,n})$ which satisfy
\begin{equation}
\label{eq:state-constraint-adjoint}
  p_{k,i-1} = -a_{k,i-1} + D_qF_{i-1}(q_{k,i-1}, u_{k,i-1})^* p_{k,i}
\end{equation}
and for all v $\in T_{\U_i}^C(u_{k,i})$
\begin{align}
\label{eq:state-constraint-pmp}
-\Delta_k \left\|v \right\|_g \le \left<b_{k,i-1} -D_u F_{i-1}(q_{k,i-1},u_{k,i-1})^*p_{k,i} , v \right> \hspace{10pt} (1 \le i \le n-1).
\end{align}

Since the sets $S_i$ are $\varphi_i$-regular along $\s{u}$ we may pass to a subsequence for which the sequences $(\s{a}_k)_{k = 1}^\infty$ and $(\s{b}_k)_{k = 1}^\infty$ converge to sequences $\s{a}$ and $\s{b}$ which satisfy $b_{0} \in \partial_{L,q} \varphi_0(e,u_0)$,  $(a_{i},b_{i}) \in \partial_{L,(q,u)} \varphi_i(e,q_{i},u_{i})$, and $a_{n} \in \partial_{L,q} \varphi_n(e,q_{n})$. Moreover, because $P(e_k,\s{u}_k) > 0$ we may choose the subsequence to insure that at least one element of either $\s{a}$ or $\s{b}$ has norm bounded below by $\delta > 0$.

 Taking the limit in \eqref{eq:state-constraint-adjoint} we obtain a costate sequence $\s{p}$ which satisfies \eqref{eq:normal-adjoint}. Because the sets $\U_i$ are Clarke regular, we may take the limit in \eqref{eq:state-constraint-pmp} and obtain \eqref{eq:normal-max}. Strict $\varphi_i$-normality of control $\s{u}$ now requires that $\s{a}$ and $\s{b}$ are identically equal to zero and this is a contradiction. Consequently $P$ must satisfy the strong decrease condition near $\ol{\s{u}} \in \cA(\ol{e})$.
\end{proof}

The following corollary is now immediate from Theorem \ref{thm:penalty} and Theorem \ref{thm:strong-decrease-necessary}:
\begin{corollary}
\label{cor:penalization-for-mixed}
 If the sets $\U_i$ are Clarke regular, constraints \eqref{eq:abstract-mixed-constraints} $\varphi_i$-regular along $\ol{\s{u}} \in \cA(\ol{e})$, and control $\ol{\s{u}}$ strictly $\varphi_i$-normal, then there exist $\ve, \kappa > 0$ such that for all $(e,\s{u}) \in E \times \cU$ satisfying $d_E(e, \ol{e})+d_{\cU}(\s{u}, \ol{\s{u}}) < \ve$ there holds
  \begin{equation*}
    d_{\cA(e)}{\s{u}} \le \kappa P(e, \s{u}).
  \end{equation*}
  In particular, the sets $\cA(e)$ are nonempty for $d_E(e,\ol{e}) < \ve$.
\end{corollary}

\subsection{Sensitivity}

 We mention two consequences Corollary \ref{cor:penalization-for-mixed} related to the value function $v : E \rightarrow \R \cup \left\{\infty\right\}$, which is defined as usual through
 \begin{equation*}
  v(e) = \inf \left\{ J(\s{u}) \, : \, \s{u} \in \cA(e) \right\}.
\end{equation*}
The results in this section are of a very classical flavor (see, for example, Clarke \cite[Section 6.4]{clarke1983}).

\begin{corollary}
  \label{cor:finite-value-function}
Suppose that control $\s{u} \in \cU$ is optimal for $J$ subject to the constraint $\s{u} \in \cA(0)$, so that $J(\s{u}) = v(0)$. If sets $\U_i$ are Clarke regular, constraints \eqref{eq:abstract-mixed-constraints} are $\varphi_i$-regular along $\s{u}$, and control $\s{u}$ strictly $\varphi_i$-normal then $v$ locally finite at $e=0$.
\end{corollary}

\begin{proof}
By Corollary \ref{cor:penalization-for-mixed}, the sets $\cA(e)$ are nonempty for $e$ in a neighborhood of $0$.
\end{proof}

\begin{theorem}
\label{thm:calm-value}
In addition to the assumptions of Corollary \ref{cor:finite-value-function}, suppose that $E$ is a closed subset of $\R^m$ and the functions $J$ and $P$ are globally Lipschitz. Then $v$ is \emph{calm} at $0 \in \R^m$ in the sense that
\begin{equation}
\label{eq:defn-calm}
-\infty < \liminf_{e \rightarrow 0} \frac{v(e) - v(0)}{\left\|e\right\|}.
\end{equation}
\end{theorem}

\begin{proof}
Let $(e_i)_{i = 0}^\infty$ be a sequence converging to $0 \in \R^m$. Since $v(e_i)$ is eventually finite we may choose for each $i$ a control $\s{u}_i \in \cA(e_i)$ such that $J(\s{u}_i) \le v(e_i) + \left\|e_i \right\|^2$. Let $K_J$ be a Lipschitz constant for $J$ and choose $\kappa > K_J$. Using exact penalization we can see that $J(\s{u}_i) +  \kappa d_{\cA(0)}(\s{u}_i) \ge v(0)$. From this and our choice of $\s{u}_i$ we obtain
\begin{align*}
  \frac{v(e_i) - v(0)}{\left\|e\right\|} & \ge \frac{J(\s{u}_i) -  \left\|e_i \right\|^2 - v(0)}{\left\|e_i\right\|} \ge  -\left\|e_i\right\| - \frac{\kappa d_{\cA(0)}(\s{u}_i)}{\left\|e_i\right\|}.
\end{align*}

Now use Corollary \ref{cor:penalization-for-mixed} and $P(e_i,\s{u}_i) = 0$ to write
\begin{align*}
  \frac{v(e_i) - v(0)}{\left\|e\right\|} & \ge - \left\|e_i\right\| + \kappa \frac{P(e_i,\s{u}_i) - P(0, \s{u}_i)}{\left\|e_i\right\|},
\end{align*}
for a possibly larger value of $\kappa$ and for $i$ sufficiently large. Since $P$ is Lipschitz and sequence $(e_i)_{i = 1}^\infty$ is arbitrary, \eqref{eq:defn-calm} follows.
\end{proof}

\subsection{Application to Smooth Equality and Inequality Constraints}

Suppose we are interested in minimizing a function $J : \cU \rightarrow \R$ as in \eqref{eq:jdef}, subject to the constraints
\begin{equation*}
\left\{ \begin{aligned}
g_j(q_i,u_i)& \le e_j \hspace{20pt} (1 \le i \le n-1, \; 1 \le j \le r) \\
h_j(q_i,u_i)& = e_{r+j} \hspace{10pt} (1 \le i \le n-1, \; 1 \le j \le s)  \end{aligned} \right.
\end{equation*}
as well as constraints
\begin{equation*}
\left\{ \begin{aligned}
G_j(q_n)& \le e_j \hspace{20pt} (1 \le j \le r) \\
H_j(q_n)& = e_{r+j} \hspace{10pt} (1 \le j \le s). \end{aligned} \right.
\end{equation*}

Suppose that we have a control $\s{u}$ which minimizes $J$ for $e = 0$, so that $J(\s{u}) = v(0)$. Suppose also that for each $1 \le i \le n-1$ the only solution to the following problem:
\begin{equation*}
  \sum_{j = 1}^r \lambda_j dg_j(q_i, u_i) + \sum_{j = 1}^s \mu_j h_j(q_i,u_i) = 0
\end{equation*}
satisfying the nonnegativity condition $\lambda_j \ge 0$ and the complementary slackness condition $\lambda_j g_j(q_i,u_i) = 0$ is $(\lambda, \mu) = 0$. Under the analogous assumption on $dG_j(q_n)$ and $dH_j(q_n)$, Lemma \ref{lem:equality-inequality-regular} assures us that the constraints are $\varphi_i$-regular along $\s{u}$ for functions $\varphi_i$ defined as in \eqref{eq:max-type-penalty-function}. Thus if control sets $\U_i$ are Clarke regular then one of the following must hold. First, it may be that $\s{u}$ is not strictly $\varphi_i$-normal. In this case there exists a costate sequence $(p_i)_{i = 1}^n$ and a nonzero sequence $(\lambda_{j,i}, \mu_{j,i})_{i =1, j = 1}^{n,r}$ such that
\begin{equation*}
  -p_n = \sum_{j = 1}^r \lambda_{n,j} dG_j(q_n) + \sum_{j = 1}^s \mu_{n,j} dH_j(q_n)
\end{equation*}
and
\begin{align*}
  p_{i-1}  &= -\sum_{j = 1}^r \lambda_{i-1,j} d_qg_j(q_{i-1}, u_{i-1}) - \sum_{j = 1}^s \mu_{i-1,j} d_q h_j(q_{i-1}, u_{i-1}) \\
	& \quad + D_qF_{i-1}(q_{i-1}, u_{i-1})^* p_{k,i}
\end{align*}
and for all $v \in T_{\U_i}^C(u_i)$
\begin{align*}
& \left< \sum_{j = 1}^r \lambda_{i-1,j} d_ug_j(q_{i-1}, u_{i-1}) + \sum_{j = 1}^s \mu_{i-1,j} d_u h_j(q_{i-1}, u_{i-1})-D_u F_{i-1}(q_{i-1},u_{i-1})^*p_{i},v\right>\\
& \ge 0 \hspace{10pt} (1 \le i \le n-1).
\end{align*}
For each $i$ both the complementary slackness and positivity conditions hold on $\lambda$. Moreover, we must have either $\s{p} \ne 0$ or else there exists an index $i$ for which
\begin{equation*}
  \sum_{j = 1}^r \lambda_{i-1,j} d_ug_j(q_{i-1}, u_{i-1}) + \sum_{j = 1}^s \mu_{i-1,j} d_u h_j(q_{i-1}, u_{i-1}) \ne 0.
\end{equation*}

On the other hand $\s{u}$ may be strictly $\varphi_i$-normal. In this case the feasible sets are nonempty for sufficiently small perturbations $e$ of the right-hand side. If in addition one can show that functions $\varphi_i$ are globally Lipschitz then the value function is calm at $0$. This is the case, for example, if $Q$ and $U_i$ are compact and $e$ is restricted to a compact neighborhood of the origin in $\R^{r+s}$.

We turn now to the study of necessary optimality conditions for constrained problems.

\section{Discrete-Time Geometric Maximum Principle for Constrained Pro-blems}
\label{sec:pmp-constraints}

In this section we develop the maximum principle for the problem of minimizing a function $J$ defined by \eqref{eq:jdef} subject to control constraints $u_i \in \U_i \subset U$ and constraints of the following types:
\begin{enumerate}[(i)]
\item Pure state constraints: $q_i \in S_i \subset Q$;
\item Mixed constraints: $(q_i, u_i) \in S_i \subset Q \times U$ for $1 \le i \le n-1$, $q_n \in S_n \subset Q$.
\end{enumerate}
We consider first the case of pure state constraints.

\subsection{Pure State Constraints}

For the sake of exposition we will assume in this section that $\ell$ and $L_i$ are $C^1$-smooth functions. This simplifies the statements of the theorems and covers a great deal of applications. We emphasize, however, that Lipschitz costs may be covered using exactly the same techniques.

\begin{theorem}
\label{thm:pure-state-max-princ}
  Suppose that $\s{u} \in \cU$ minimizes a cost function $J$ defined by \eqref{eq:jdef} subject to pointwise state constraints
  \begin{equation*}
    q_i \in S_i \colonequals \left\{q \in Q \, : \, \varphi_i(q) = 0 \right\}
  \end{equation*}
  for locally Lipschitz, nonnegative functions $\varphi_i$. If the sets $\U_i$ are Clarke regular, and functions $\ell$, $L_i$ are $C^1$-smooth, and the constraints are $\varphi_i$-regular along $\s{u}$ then there is a number $\lambda_0 \in \left\{0,1\right\}$ and a costate sequence $\s{p}$ which:
   \begin{enumerate}[(i)]
   \item Satisfies the endpoint condition: $- p_n \in \lambda_0 d \ell(q_n) + \partial_L \varphi_n(q_n)$;
\item Evolves according to:
  \begin{equation}
  \label{eq:abstract-adjoint-1}
    p_{i-1} + D_q F_{i-1}(q_{i-1}, u_{i-1})^* p_i \in \lambda_0 d_q L_{i-1}(q_{i-1}, u_{i-1}) + \pL \varphi_{i-1}(q_{i-1});
  \end{equation}
\item And satisfies for all $v \in T_{\U_{i-1}}^C(u_{i-1})$
  \begin{align}
  \label{eq:sc-max-princ}
  0 \le \left<\lambda_0d_u L_{i-1}(q_{i-1}, u_{i-1}) - D_u F_{i-1}(q_{i-1}, u_{i-1})^* p_i , v\right> \hspace{10pt} (1 \le i \le n).
  \end{align}
\end{enumerate}
Moreover, either $\lambda_0 = 1$ or $\s{p}$ is not identically equal to zero.
\end{theorem}

\begin{proof}
We consider two possibilities: either $\s{u}$ is strictly $\varphi_i$-normal or it is not. If $\s{u}$ is not strictly $\varphi_i$-normal then the theorem holds with $\lambda_0 = 0$, since in this case the conclusions are a restatement of Definition \ref{defn:strictly-normal}. We need only check that $\s{p}$ may be chosen so that it has at last one nonzero entry. This is a straightforward consequence of Definition \ref{defn:strictly-normal}, the endpoint condition, and \eqref{eq:abstract-adjoint-1}.

On the other hand, if $\s{u}$ is strictly $\varphi_i$-normal then there exist real numbers $\ve, \kappa > 0$ such that
\begin{equation*}
  d_{\cA}(\s{c}) \le \kappa \sum_{i = 1}^n \varphi_i(r_i)
\end{equation*}
for all controls $\s{c}$ within distance $\ve$ of $\s{u}$, where we have written $\s{r}$ for the state sequence corresponding to control $\s{c}$. Again using an exact penalization argument and possibly increasing $\kappa$ we can see that $\s{u}$ is a local minimizer of the function
\begin{equation*}
  \Lambda(\s{c}) \colonequals J(\s{c}) + \kappa \sum_{i = 1}^n \varphi_i(r_i).
\end{equation*}
Thus $\s{u}$ is critical for $\Lambda$ and the result now follows with $\lambda_0 = 1$.
\end{proof}

As in Section \ref{sec:discrete-gmp} we may make additional assumptions to arrive at a maximum principle.

\begin{theorem}
  Suppose that $F_i$ and $\U_i$ satisfy assumptions \textnormal{\textbf{(}}$\mathsf{A}$\textnormal{\textbf{)}} of Section \ref{sec:discrete-gmp}, the constraints are $\varphi_i$-regular along $\s{u}$, and functions $\ell$, $L_i$ are $C^1$-smooth. Then in the conclusions of Theorem \ref{thm:pure-state-max-princ} we may replace \eqref{eq:sc-max-princ} with
\begin{equation*}
  H_i(\lambda_0,u_i,p_{i+1}) = \max_{u \in \U_i}   H_i(\lambda_0,u_i,p_{i+1}) \hspace{10pt} (0 \le i \le n-1)
\end{equation*}
where $H_i : \left\{0,1\right\} \times \U_i \times T_{q_{i+1}}^* Q \rightarrow \R$ by
\begin{equation}
\label{eq:hamiltonian-constraints}
H_i(\lambda,u,p) \colonequals \left<\mathbb{F}E_{f_i(q_i,u_i)}^* p, f(q_i,u) \right> - \lambda L_i(q_i,u).
\end{equation}
  \end{theorem}

  \begin{proof}
    The details follow those of Theorems \ref{thm:discrete-geometric-pmp} and \ref{thm:pure-state-max-princ}. We note that the assumption of Clarke regularity, which is necessary for Theorem \ref{thm:pure-state-max-princ}, is automatic from the convexity of $\U_i$.
  \end{proof}

We emphasize that every closed set is $d_S$-regular and so, if one is willing to accept results in terms of the normal cone, then \eqref{eq:abstract-adjoint-1} can be written as
 \begin{equation*}
    p_{i-1} + D_q F_{i-1}(q_{i-1}, u_{i-1})^* p_i \in \lambda_0 d_q L_{i-1}(q_{i-1}, u_{i-1}) + N_{S_{i-1}}^L(q_{i-1})
  \end{equation*}
 without any assumptions on the constraint sets $S_i$ beyond asking that they be closed.

\subsection{Mixed Constraints}
\label{sec:mixed-sensitivity}

As before we assume in this section that $\ell$ and $L_i$ are $C^1$-smooth functions in order to simplify statements of Theorem and remark that Lipschitz costs may be covered using exactly the same techniques. We consider the following problem: \emph{Minimize $J : \cU \rightarrow \R$ defined by \eqref{eq:jdef} subject to constraints}
  \begin{equation}
  \label{eq:constant-mixed-constraints}
\left\{ \begin{aligned}
(q_i,u_i) & \in S_i \hspace{2pt} \colonequals \left\{(q,u) \in Q \times U_i \, : \, \varphi_i(q,u) = 0 \right\} \hspace{10pt} (1 \le i \le n-1) \\
q_n & \in S_n \colonequals \left\{q \in Q  \, : \, \varphi_n(q) = 0 \right\}  \end{aligned} \right.
  \end{equation}
\emph{where functions $\varphi_i$ are locally Lipschitz and nonnegative}.

\begin{theorem}
\label{thm:mixed-max-princ}
  Suppose that $\s{u} \in \cU$ minimizes a cost function $J$ defined by \eqref{eq:jdef} subject to mixed constraints \eqref{eq:constant-mixed-constraints}. If the sets $\U_i$ are Clarke regular, the constraints are $\varphi_i$-regular along $\s{u}$, and functions $F_i, \ell, L_i$ are $C^1$-smooth then there exist sequences $\s{a} = (a_i)_{i = 1}^n$ and $\s{b} = (b_i)_{i = 0}^{n-1}$ with $b_0 = 0$
  \begin{equation*}
\left\{ \begin{aligned}
(a_i,b_i) & \in \pL \varphi_i(q_i,u_i) \hspace{10pt} (1 \le i \le n-1) \\
a_n & \in \pL \varphi_n(q_n) \end{aligned} \right.
  \end{equation*}
for which the costate sequence $\s{p}$ defined by
\begin{equation}
\label{eq:mixed-constraints-endpoint}
p_n = -a_n - \lambda_0 d \ell(q_n)
 \end{equation}
 and
 \begin{equation}
\label{eq:mixed-constraints-adjoint}
   p_{i-1} = \lambda_0 d_q L_{i-1}(q_{i-1}, u_{i-1}) + a_i - D_q F_{i-1}(q_{i-1}, u_{i-1})^* p_i
 \end{equation}
 satisfies for all $v \in T_{\U_{i-1}}^C(u_{i-1})$
 \begin{align*}
0 \le \left< \lambda_0 dL_{i-1}(q_{i-1}, u_{i-1}) + b_i - D_u F_{i-1}(q_{i-1}, u_{i-1})^*p_i , v\right> \hspace{10pt} (1 \le i \le n).
 \end{align*}
 Moreover, either $\s{a}$ and $\s{b}$ are not both identically zero or else $\lambda_0 = 1$.
\end{theorem}

\begin{proof}
If $\s{u}$ is not strictly $\varphi_i$-normal, then result follows from the definition of strict $\varphi_i$-normality. On the other hand, if $\s{u}$ is strictly $\varphi_i$-normal then there exist real numbers $\ve, \kappa > 0$ such that
\begin{equation*}
  d_{\cA}(\s{c}) \le \kappa \varphi_n(r_n) + \kappa \sum_{i = 1}^n \varphi_i(r_i, c_i)
\end{equation*}
for all $\s{c}$ within distance $\ve$ of $\s{u}$. Using a standard exact penalization argument and possibly increasing $\kappa$ we can see that $\s{u}$ is a local minimizer of the function
\begin{equation*}
  \Lambda(\s{c}) \colonequals J(\s{c}) + \kappa \varphi_n(r_n) + \kappa \sum_{i = 1}^n \varphi_i(r_i, c_i).
\end{equation*}
Thus $\s{u}$ is extremal for $\Lambda$ and the result now follows with $\lambda_0 = 1$.
\end{proof}

Theorem \ref{thm:mixed-max-princ} is lacking the usual nondegeneracy condition on $\s{p}$ and the conditions on $\s{a}$ and $\s{b}$ may not be very informative. By strengthening the assumptions on the sets $S_i$ we can obtain a nondegeneracy result for $\s{p}$:
\begin{corollary}
Suppose that in addition to the hypotheses of Theorem \ref{thm:mixed-max-princ} we suppose that the sets $S_i$ for $1 \le i \le n-1$ satisfy, for some $\kappa > 0$, the following \emph{bounded slope} condition:
\begin{equation}
\label{eq:bounded-slope-condition}
  (a,b) \in \partial_L \varphi_i(q,u) \Rightarrow \left\|b\right\| \le \kappa \left\|a \right\|.
\end{equation}
Then either $\s{p}$ is nonzero or $\lambda_0 = 1$.
\end{corollary}

\begin{proof}
In the case where $\lambda_0 = 0$ we must have $\s{a} \ne 0$, for otherwise the bounded slope condition will force both $\s{a}$ and $\s{b}$ to be zero. We can then see that $\s{p}$ is nonzero from \eqref{eq:mixed-constraints-endpoint} and \eqref{eq:mixed-constraints-adjoint}.
\end{proof}

It's worth comparing the bounded slope condition given by \eqref{eq:bounded-slope-condition} with that found in the later chapters of \cite{clarke2013}, the source from which we borrow the name.

\subsection{Discussion}

We conclude with an analysis of an abstract discrete-time optimal control problem in order to demonstrate some of the above theorems. Let us suppose we are interested in minimizing a function
\begin{equation*}
  J(\s{u}) = \ell(q_n) + \sum_{i = 1}^{n-1} L_i(q_i, u_i)
\end{equation*}
subject to mixed constraints
\begin{equation*}
\left\{ \begin{aligned}
g_j(q_i,u_i) &\le 0 \hspace{10pt} (1 \le i \le n-1, \; 1 \le j \le r) \\
h_j(q_i,u_i) &= 0 \hspace{10pt} (1 \le i \le n-1, \; 1 \le j \le s) \end{aligned} \right.
\end{equation*}
and endpoint constraints
\begin{equation*}
\left\{ \begin{aligned}
G_j(q_n) &\le 0 \hspace{10pt} (1 \le j \le r) \\
H_j(q_n) & \le 0 \hspace{10pt} (1 \le j \le s). \end{aligned} \right.
\end{equation*}
We suppose the all of the above functions are $C^1$-smooth and the the controls are subject only to the mixed constraints, so that $\U_i \colonequals U_i$. Thus the sets $\U_i$ are all Clarke regular.

Suppose that control $\s{u}$ is optimal for the problem and that for each $1 \le i \le n-1$, the only solution to
\begin{equation*}
  \sum_{j = 1}^r \lambda_j dg_j(q_i, u_i) + \sum_{j = 1}^s \mu_j dh_j(q_i,u_i) = 0
\end{equation*}
satisfying the nonnegativity condition $\lambda_j \ge 0$ and the complementary slackness condition $\lambda_j g_j(q_i,u_i) = 0$, each for $1 \le j \le r$, is $\lambda_j, \mu_j \equiv 0$. We make the analogous assumption for constraints $G_j$ and $H_j$.

In this case the constraints are $\varphi_i$-regular along $\s{u}$ and so by Theorem \ref{thm:mixed-max-princ} we can find a sequence $(\lambda_{j,i},\mu_{j,i})_{i = 1, j = 1}^{n, r}$ satisfying the nonnegativity condition $\lambda_{j,i} \ge 0$ and the complementary slackness conditions
\begin{equation*}
\left\{ \begin{aligned}
\lambda_{j,i} g_j(q_i,u_i) & = 0 \hspace{10pt} (1 \le i \le n-1,\; 1 \le j \le r)\\
\lambda_{n,i} G_j(q_i) & = 0 \hspace{10pt} (1 \le j \le r) \end{aligned}\right.
\end{equation*}
along with a number $\lambda_0 \in \left\{0,1\right\}$ such that the costate sequence defined by
\begin{equation*}
  -p_n \colonequals \lambda_0 d\ell(q_n) + \sum_{j = 1}^r \lambda_{n,j} dG_j(q_n)
\end{equation*}
and
\begin{align*}
  p_{i-1} & = - \lambda_0 d_q L_{i-1}(q_{i-1}, u_{i-1}) - \sum_{j = 1}^r \lambda_{j,i-1} d_q g_j (q_{i-1}, u_{i-1}) \\
   & \quad - \sum_{j = 1}^s \mu_{j,i-1} d_q h_j (q_{i-1}, u_{i-1}) + D_qF_{i-1}(q_{i-1}, u_{i-1})^* p_i
\end{align*}
satisfies for all $v \in T_{u_{i}} U_i$
\begin{align*}
  &\left< \lambda_0 d_u L_{i}(q_{i}, u_{i})+\sum_{j = 1}^r \lambda_{j,i}d_u g_{j}(q_{i}, u_{i}) + \sum_{j = 1}^s \mu_{j,i}d_u h_{j}(q_{i}, u_{i})- D_uF_{i}(q_{i}, u_{i})^* p_{i+1},v\right>\\
	& \quad \ge 0 \hspace{10pt} (0 \le i \le n-1).
\end{align*}
This last inequality forces
\begin{align*}
 D_uF_{i-1}(q_{i-1}, u_{i-1})^* p_i &= \lambda_0 d_u L_{i-1}(q_{i-1}, u_{i-1})+\sum_{j = 1}^r \lambda_{j,i-1}d_u \varphi_{i-1}(q_{i-1}, u_{i-1}) \\
 & \quad + \sum_{j = 1}^s \mu_{j,i}d_u h_{j}(q_{i}, u_{i})
\end{align*}
Moreover, either $\lambda_0 = 1$, $p_n \ne 0$, or for some $1 \le j \le n-1$ we have
\begin{equation*}
  \sum_{j = 1}^r \lambda_{j,i} dg_j(q_i,u_i) + \sum_{j = 1}^r \mu_j dh_j(q_i, u_i) \ne 0.
\end{equation*}
If in addition the constraints can be shown to satisfy the bounded slope condition:
\begin{align*}
&  \left\|\sum_{j = 1}^r \lambda_{j,i} d_u g_j(q_i,u_i) + \sum_{j =1}^s \mu_{j,i} d_u h_j(q_i, u_i) \right\| \\
&  \quad \le \kappa \left\|\sum_{j = 1}^r \lambda_{j,i} d_q g_j(q_i,u_i) + \sum_{j =1}^s \mu_{j,i} d_q h_j(q_i, u_i) \right\|
\end{align*}
for some $\kappa > 0$ and for $1 \le i \le n-1$ then we may suppose that either $\lambda_0 = 1$ or else $\s{p} \not \equiv 0$.

The results of Section \ref{sec:mixed-sensitivity} can be applied to study the sensitivity of this problem to perturbations in the constraints. Finally, we remark that in many cases much more can be said using the results of this paper. For example, if the problem evolves on a Lie group according to $g_{i} \colonequals g_{i-1} u_{i-1}$ then the techniques of Section \ref{sec:liegroups} may be applied and if $F_i$ factors as \eqref{eq:factorization} then in many applications $D_qF_i^*$ can be profitably expressed in terms of the fibre derivative. In this case, additional assumptions will lead in a natural way to a maximum principle as in \eqref{eq:max-princ-fiber-version} and \eqref{eq:hamiltonian-constraints}.

\section*{Acknowledgment}
The work presented in this paper was carried out while the second author was a postdoctoral fellow at the Institute for Mathematics and its Applications (IMA) during the IMA's annual program on \textit{``Control Theory and its Applications''}.

\bibliographystyle{abbrv}
\bibliography{dgmpbib}

\begin{thebibliography}{10}

\bibitem{absil2007optimization}
P.-A. Absil, R.~Mahony, and R.~Sepulchre.
\newblock {\em Optimization Algorithms on Matrix Manifolds}.
\newblock Princeton University Press, Princeton, NJ, 2007.

\bibitem{AS2004}
A.~A. Agrachev and Y.~L. Sachkov.
\newblock {\em Control Theory from the Geometric Viewpoint}.
\newblock Springer-Verlag, Berlin Heidelberg, Germany, 2004.

\bibitem{azagra2007applications}
D.~Azagra and J.~Ferrera.
\newblock Applications of proximal calculus to fixed point theory on
  \uppercase{R}iemannian manifolds.
\newblock {\em Nonlinear Analysis: Theory, Methods \& Applications},
  67(1):154--174, 2007.

\bibitem{barbero2009}
M.~Barbero-Li{\~n}{\'a}n and M.~C. Mu{\~n}oz-Lecanda.
\newblock Geometric approach to \uppercase{P}ontryagin's maximum principle.
\newblock {\em Acta Applicandae Mathematicae}, 108(2):429--485, 2009.

\bibitem{bloch2003}
A.~M. Bloch.
\newblock {\em Nonholonomic Mechanics and Control}.
\newblock Springer-Verlag, New York, NY, 2003.

\bibitem{boltyanskii1978}
V.~G. Boltyanskii.
\newblock {\em Optimal Control of Discrete Systems}.
\newblock John Wiley \& Sons, New York, NY, 1978.

\bibitem{borweinzhu2005}
J.~Borwein and Q.~J. Zhu.
\newblock {\em Techniques of Variational Analysis}.
\newblock CMS Books in Mathematics. Springer, New York, NY, 2005.

\bibitem{bourdin2013}
L.~Bourdin and E.~Tr{\'e}lat.
\newblock Pontryagin maximum principle for finite dimensional nonlinear optimal
  control problems on time scales.
\newblock {\em SIAM Journal on Control and Optimization}, 51(5):3781--3813,
  2013.

\bibitem{bullo-lewis2005}
F.~Bullo and A.~D. Lewis.
\newblock {\em Geometric Control of Mechanical Systems: Modeling, Analysis, and
  Design for Simple Mechanical Control Systems}.
\newblock Springer-Verlag, New York, NY, 2005.

\bibitem{clarke1983}
F.~H. Clarke.
\newblock {\em Optimization and Nonsmooth Analysis}.
\newblock John Wiley \& Sons, New York, NY, 1983.
\newblock (Reprinted as vol. 5 of Classics in Applied Mathematics, SIAM,
  Philadelphia, PA, 1990).

\bibitem{clarke2013}
F.~H. Clarke.
\newblock {\em Functional Analysis, Calculus of Variations and Optimal
  Control}.
\newblock Springer, New York, NY, 2013.

\bibitem{clarkeand1998}
F.~H. Clarke, Y.~S. Ledyaev, R.~J. Stern, and P.~R. Wolenski.
\newblock {\em Nonsmooth Analysis and Control Theory}.
\newblock Springer-Verlag, New York, NY, 1998.

\bibitem{de1980problems}
E.~De~Giorgi, A.~Marino, and M.~Tosques.
\newblock Problems of evolution in metric spaces and maximal decreasing curve.
\newblock {\em Atti Accad. Naz. Lincei Rend. Cl. Sci. Fis. Mat. Natur. (8)},
  68(3):180--187, 1980.
\newblock (In Italian).

\bibitem{guler2010foundations}
O.~G{\"u}ler.
\newblock {\em Foundations of Optimization}.
\newblock Springer-Verlag, New York, NY, 2010.

\bibitem{gupta2015}
R.~Gupta, U.~V. Kalabi{\'c}, S.~Di~Cairano, A.~M. Bloch, and I.~V. Kolmanovsky.
\newblock Constrained spacecraft attitude control on \uppercase{SO}$($3$)$
  using fast nonlinear model predictive control.
\newblock In {\em Proceedings of American Control Conference}, pages
  2980--2986, 2015.

\bibitem{iserles2000lie}
A.~Iserles, H.~Z. Munthe-Kaas, S.~P. N{\o}rsett, and A.~Zanna.
\newblock Lie-group methods.
\newblock {\em Acta Numerica 2000}, 9:215--365, 2000.

\bibitem{jurdjevic1997geometric}
V.~Jurdjevic.
\newblock {\em Geometric Control Theory}.
\newblock Cambridge University Press, Cambridge, United Kingdom, 1997.

\bibitem{kalabic2016}
U.~V. Kalabi{\'c}, R.~Gupta, S.~Di~Cairano, A.~M. Bloch, and I.~V. Kolmanovsky.
\newblock M\uppercase{PC} on manifolds with an application to the control of
  spacecraft attitude on \uppercase{SO}$($3$)$.
\newblock {\em Automatica}, 2016.
\newblock (To Appear).

\bibitem{kipka-ledyaev-2014a}
R.~J. Kipka and Y.~S. Ledyaev.
\newblock Optimal control on manifolds: optimality conditions via nonsmooth
  analysis.
\newblock {\em Communications in Applied Analysis}, 18:563--590, 2014.

\bibitem{kipka-ledyaev-2015a}
R.~J. Kipka and Y.~S. Ledyaev.
\newblock Pontryagin maximum principle for control systems on infinite
  dimensional manifolds.
\newblock {\em Set-Valued and Variational Analysis}, 23(1):133--147, 2015.

\bibitem{kobilarov2007}
M.~B. Kobilarov, M.~Desbrun, J.~E. Marsden, and G.~S. Sukhatme.
\newblock A discrete geometric optimal control framework for systems with
  symmetries.
\newblock In {\em Proceedings of Robotics: Science and Systems}, 2007.

\bibitem{kobilarov2011}
M.~B. Kobilarov and J.~E. Marsden.
\newblock Discrete geometric optimal control on \uppercase{L}ie groups.
\newblock {\em IEEE Transactions on Robotics}, 27(4):641--655, 2011.

\bibitem{ledyaevzhu2007}
Y.~S. Ledyaev and Q.~J. Zhu.
\newblock Nonsmooth analysis on smooth manifolds.
\newblock {\em Transactions of the American Mathematical Society},
  359(8):3687--3732, 2007.

\bibitem{lee2005lie}
T.~Lee, N.~H. McClamroch, and M.~Leok.
\newblock A \uppercase{L}ie group variational integrator for the attitude
  dynamics of a rigid body with applications to the 3\uppercase{D} pendulum.
\newblock In {\em Proceedings of Conference on Control Applications}, pages
  962--967, 2005.

\bibitem{marsden2001}
J.~E. Marsden and M.~West.
\newblock Discrete mechanics and variational integrators.
\newblock {\em Acta Numerica}, 10(1):357--514, 2001.

\bibitem{mordukhovich2006a}
B.~S. Mordukhovich.
\newblock {\em Variational Analysis and Generalized Differentiation I: Basic
  Theory}.
\newblock Springer-Verlag, Berlin Heidelberg, Germany, 2006.

\bibitem{mordukhovich2006b}
B.~S. Mordukhovich.
\newblock {\em Variational Analysis and Generalized Differentiation II:
  Applications}.
\newblock Springer-Verlag, Berlin Heidelberg, Germany, 2006.

\bibitem{moser1991}
J.~Moser and A.~P. Veselov.
\newblock Discrete versions of some classical integrable systems and
  factorization of matrix polynomials.
\newblock {\em Communications in Mathematical Physics}, 139(2):217--243, 1991.

\bibitem{munthe1999high}
H.~Munthe-Kaas.
\newblock High order \uppercase{R}unge-\uppercase{K}utta methods on manifolds.
\newblock {\em Applied Numerical Mathematics}, 29(1):115--127, 1999.

\bibitem{penot2013}
J.-P. Penot.
\newblock {\em Calculus without Derivatives}.
\newblock Springer-Verlag, New York, NY, 2013.

\bibitem{phogat2015}
K.~S. Phogat, D.~Chatterjee, and R.~N. Banavar.
\newblock Discrete-time optimal attitude control of spacecraft with momentum
  and control constraints.
\newblock {\em ArXiv e-prints}, 2015.

\bibitem{pontryagin1962}
L.~S. Pontryagin, V.~G. Boltyanskii, R.~V. Gamkrelidze, and E.~F. Mishchenko.
\newblock {\em The Mathematical Theory of Optimal Processes}.
\newblock John Wiley \& Sons, New York, NY, 1962.

\bibitem{schirotzek2007}
W.~Schirotzek.
\newblock {\em Nonsmooth Analysis}.
\newblock Springer-Verlag, Berlin Heidelberg, Germany, 2007.

\bibitem{subbotin1991property}
A.~Subbotin.
\newblock On a property of the subdifferential.
\newblock {\em Matematicheskii Sbornik}, 182(9):1315--1330, 1991.
\newblock (In Russian).

\bibitem{sussmann1997}
H.~J. Sussmann.
\newblock An introduction to the coordinate-free maximum principle.
\newblock In {\em Geometry of Feedback and Optimal Control, \textnormal{B.
  Jakubczyk and W. Respondek Eds., Marcel Dekker, New York}}, pages 463--557,
  1997.

\end{thebibliography}

\end{document}